\documentclass[11pt]{article}
\usepackage{mathrsfs}
\usepackage{amsfonts}
\usepackage{}
\usepackage{amsmath,amssymb,amsthm,latexsym,amstext}
\usepackage[mathscr]{eucal}
\usepackage{rotate,graphics,epsfig,epstopdf}
\usepackage{float}
\usepackage{color}
\usepackage{subfigure}
\usepackage{indentfirst}
\usepackage{bm}

\newcommand{\be}{\begin{eqnarray}}
\newcommand{\ee}{\end{eqnarray}}
\newcommand{\by}{\begin{eqnarray*}}
\newcommand{\ey}{\end{eqnarray*}}
\newcommand{\bn}{\begin{enumerate}}
\newcommand{\en}{\end{enumerate}}
\newcommand{\ei}{\end{itemize}}

\newtheorem{theorem}{Theorem}
\newtheorem{lemma}[theorem]{Lemma}
\newtheorem{assumption}{Assumption}

\newtheorem{remark}[theorem]{Remark}
\newtheorem{definition}[theorem]{Definition}

\renewcommand{\theequation}{\arabic{section}.\arabic{equation}}

\numberwithin{equation}{section}

\begin{document}
\date{}
\title{\bf The Smoluchowski-Kramers approximation for a system with  arbitrary friction depending on both state and distribution \footnote{This work was supported by
the National Natural Science Foundation of China  NSFC No.11771207.
}}
\author{ Xueru Liu\footnote{Department of Mathematics, Nanjing University,
Nanjing, China. dg21210009@smail.nju.edu.cn }\hskip1cm  Qianqian Jiang\footnote{Department of Statistics and Data Science, Southern University of Science and Technology, Shenzhen, China.  Corresponding author, jiangqq@sustech.edu.cn }\hskip1cm Wei Wang\footnote{Department of Mathematics, Nanjing University,
Nanjing, China. wangweinju@nju.edu.cn} \\
}\maketitle
\begin{abstract}
 A system of stochastic differential equations describing diffusive phenomena, which has arbitrary friction depending on both  state and distribution  is investigated. The Smoluchowski-Kramers approximation is seen to describe dynamics in the small mass limit.  We obtain the limiting equation and, in particular, the addition drift terms that appear in the limiting equation are expressed in terms of the solutions to the Lyapunov matrix equation and  Sylvester matrix equation. Furthermore, we provide the rate of convergence  and extend the system to encompass  more general interactions and noise.
\end{abstract}


\section{Introduction}
  \setcounter{equation}{0}
  \renewcommand{\theequation}
{1.\arabic{equation}}

 In this paper, we are concerned with  the rigorous derivation of the mass limit for the following  $2d$-dimensional stochastic differential equation (SDE) with both  state-dependent and distribution-dependent  damping:
\begin{eqnarray}\label{maineq}
\left\{\begin{array}{ll}
  {d} \bm{x}_t^{\epsilon} = \bm{v}_{t}^{\epsilon} {d}t,\\
  {d}\bm{v}_{t}^{\epsilon} = \frac{1}{\epsilon}F(\bm{x}_t^{\epsilon}) {d}t-\frac{1}{\epsilon}\bm{\gamma}( \bm{x}_t^{\epsilon},\mathscr{L}_{\bm{x}_t^{\epsilon}}) \bm{v}_{t}^{\epsilon} {d}t+\frac{1}{\epsilon}\bm{\sigma}(\bm{x}_t^{\epsilon}) {d}\bm{W}_{t},
\end{array}
\right.
\end{eqnarray}
where $\bm{F}:\mathbb{R}^d\mapsto\mathbb{R}^d$  is the driving force,   $\bm{\gamma}:\mathbb{R}^d\rightarrow \mathbb{R}^{d\times d}$ is a $d\times d$ invertible matrix-valued function, and $\bm{\sigma}:\mathbb{R}^d\rightarrow \mathbb{R}^{d\times k}$ is a random external force field.   Let $\{\bm{W}_{t}\}_{t\geq0}$ be a $k$-dimensional Wiener process on a complete probability space $(\Omega,\mathcal{F},\mathbb{P})$ with a filtration $(\mathcal{F}_{t})_{t\geq 0}$ being the nature filtration generated by $\bm{W}_{t}$.  By Newton's second law of motion, the solution  of equation~(\ref{maineq}) can be interpreted as the displacement field and velocity field of the particles in a continuum body.  The physical significance of this type of models has been discussed in standard literatures, one can refer to~\cite{shamgv:smklim, Nelson,Papanicolaou}.

The study of the small mass  limit has been of interest since the
seminal work of Smoluchowski~\cite{Smoluchowski} and Kramers~\cite{Kramers},  now known as Smoluchowski--Kramers(S--K)  approximation. The mass limit of equations similar to (\ref{maineq}) are investigated by many authors, the case of a constant friction coefficient $\gamma$ is considered and  the limiting equation is formally obtained by taking $\epsilon = 0$ in (\ref{maineq})~\cite{Freidlin, Spiliopoulos, W.Wang}.
Furthermore, a substantial body of literature addresses S--K approximation in scenarios involving variable friction~\cite{MW,Hanggi, Herzog, shamgv:smklim, Pardoux, Sancho, W.Wang1}. However, the case of a state-independent friction creates an additional drift in the limiting equations which is relevant to the non-constant friction and the noises. Notably, Hottovy  et al.~\cite{shamgv:smklim}  provided a rigorous mathematical explanation, employing the theory of the convergence of stochastic
integrals with respect to semimartingales, to account for such experimental observations.

More recently, there are some work on  S--K approximation in mean filed limit for stochastic $N$-interacting
particles system~\cite[e.g.]{Carrillo, W.Wang, Xr.L}. The  motivation for our  equation~(\ref{maineq})  comes from interacting particles systems~\cite{Carrillo}, which derived the combined mean field  and small mass limit  with $\bm{\sigma}= 0$ at the hydrodynamic level by a discrete version of the modulated kinetic energy method.
More precisely, at the formal level, the $N$-point particle system of~\cite{Carrillo}  can write out the corresponding McKean--Vlasov SDE,
\begin{eqnarray}\label{intro-4}
\dot{\bm{x}_i }  &=& {\bm{v}_i },\\
  \epsilon \dot{\bm{x}_i}  &=&  - \gamma {\bm{v}_i } - \nabla_x V(\bm{x}_i)- \nabla_x W(\bm{x}_i)\star \rho_t\nonumber\\
 &&{}  +  \int_{\mathbb{R}^{d}\times \mathbb{R}^{d}} \psi( \bm{x}_i-y)w f(dy dw) -  {\bm{v}_i } \int_{\mathbb{R}^{d}\times \mathbb{R}^{d}} \psi( \bm{x}_i-y)f(dy dw)   ,\label{intro-5}
\end{eqnarray}
where $f_t(x, v)$ is the distribution of $({\bm{x}_i } ,{\bm{v}_i } )$.  For more detail one can refer to~\cite{Carrillo}.  In our case, we reset  the damping force ${\bm{v}_i } \int_{\mathbb{R}^{d}\times \mathbb{R}^{d}} \psi( \bm{x}_i-y)f(dy dw)$ in (\ref{intro-5}) to $\bm{\gamma}( \bm{x}_t^{\epsilon},\mathscr{L}_{\bm{x}_t^{\epsilon}}) \bm{v}_{t}^{\epsilon} $ and add  a random external force field. However,
such a method~\cite{Carrillo} fails due to the existence of noise.
Since the coefficients in the system \eqref{maineq} depend not only on the state $x_t^{\epsilon}$ but also on its distribution, the proof relies on a theory of convergence of stochastic integrals developed by Kurtz and Protter~\cite{Kurtz}, while also heavily hinging upon a preliminary study of the derivatives of the solution to the mean-field stochastic differential equation (also called as the McKean-Vlasov equation) with respect to the probability law and a corresponding $It\hat{o}$ formula.  We show that, under the assumptions in Section 2, the $\bm{x}_t^{\epsilon}$-component of the solution of equation~(\ref{maineq})  converges in $L^2$, with respect to the uniform norm in $C([0,T], \mathbb{R}^d)$,  to the solution $\bm{x}_{t}$  of the following limiting equation
\begin{eqnarray}\label{maineq-limit}
  {d}\bm{x}_{t} = [\bm{\gamma}^{-1}(\bm{x}_{t},\mathscr{L}_{\bm{x}_{t}})F(\bm{x}_{t})+S( \bm{x}_{t},\mathscr{L}_{\bm{x}_{t}}) +\widetilde{S}(\bm{x}_{t},\mathscr{L}_{\bm{x}_{t}})] {d}t+\bm{\gamma}^{-1}( \bm{x}_{t},\mathscr{L}_{\bm{x}_{t}}) \bm{\sigma}(\bm{x}_{t}) {d}\bm{W}_{t}
\end{eqnarray}
where $S_{i}(\bm{x},\mu) = \frac{\partial}{\partial x_{l}}[\bm{\gamma}^{-1}_{ij}(\bm{x},\mu)]J_{jl}(\bm{x},\mu)$ and $J$ solves the Lyapunov equation
\begin{eqnarray}
    J(\bm{x},\mu)\bm{\gamma}^{*}(\bm{x},\mu)+\bm{\gamma}(\bm{x},\mu)J(\bm{x},\mu) = \bm{\sigma}(\bm{x})\bm{\sigma}^{*}(\bm{x}),
\end{eqnarray}
and $\widetilde{S}_{i}(\bm{x},\mu) = \widetilde{E}[(\partial_{\mu}\bm{\gamma}^{-1}_{ij}(\bm{x},\mu)(\widetilde{\bm{x}}))_{l}\widetilde{J}_{jl}(\bm{x},\widetilde{\bm{x}},\mu)]$ and $\widetilde{J}$ solves the Sylevster equation
\begin{eqnarray}
    \widetilde{J}(\bm{x},\widetilde{\bm{x}},\mu)\bm{\gamma}^{*}(\widetilde{\bm{x}},\mu)+\bm{\gamma}(\bm{x},\mu)\widetilde{J}(\bm{x},\widetilde{\bm{x}},\mu) = \bm{\sigma}(\bm{x})\bm{\sigma}^{*}(\widetilde{\bm{x}}).
\end{eqnarray}
And the  rate of convergence is
    \begin{eqnarray}
        \mathbb{E}(\sup_{0 \leq t \leq T}\mid \bm{x}_t^{\epsilon}-\bm{x}_{t}\mid^2)\leq C\sqrt{\epsilon}.
    \end{eqnarray}

The rest of this paper is organized as follows. Some notations, assumptions and main results are introduced in Section 2. In Section 3, we study the well-posedness of equation (\ref{maineq}) and establish  some uniform bounds  for the solutions $(\bm{x}_{t}^\epsilon, \bm{v}_{t}^\epsilon)$ with respect to $\epsilon$. These
bounds enables us to demonstrate  the tightness of the distribution of $\bm{x}_{t}^\epsilon$, for sufficiently small $\epsilon$.
The limiting equation  and the rate of convergence  are derived in Section 4. Finally, we extend the system for more general interactions ${\bm F}$ and noise ${\bm\sigma}$ in Section 5.

\section{Preliminary}\label{main}
  \setcounter{equation}{0}
  \renewcommand{\theequation}
{2.\arabic{equation}}
~~~Suppose that $\mathscr{P}$ is  the set of probability measure on $(\mathbb{R}^{d},\mathscr{B}(\mathbb{R}^{d}))$. Let
\begin{eqnarray}
 \mathscr{P}_{2}:=\{\mu\in \mathscr{P}:\mu(\mid \cdot\mid^{2}):=\int_{\mathbb{R}^{d}}|x|^{2}\mu(\mathrm{d}x)<\infty\},
\end{eqnarray}
then $ \mathscr{P}_{2}$ is a polish space under the $L^{2}$-wasserstein distance, i.e.,
\begin{eqnarray}
    W(\mu,\nu):= \inf_{\rho\in\mathscr{C}_{\mu,\nu}}\Big[\Big(\int_{\mathbb{R}^{d}\times\mathbb{R}^{d}}|x-y|^{2}\rho(\mathrm{d}x,\mathrm{d}y)\Big)^{1/2}\Big], ~~\mu,\nu\in \mathscr{P}_{2},
\end{eqnarray}
where $\mathscr{C}_{\mu,\nu}$ is the set of all couplings for $\mu$ and $\nu$.
In the following parts, the notation $C_{T}$ denotes a positive constant depending on $T$ which maybe change from line to line.\\
Now, we first remind the reader of the definition of derivative on the Wasserstein space.  For the measure derivatives of introduced by Lions, we follow the approach in \cite{carp:nomefiga}. Consider $u:\mathscr{P}_{2}\rightarrow \mathbb{R}$.

\begin{definition}{\bf ($L$-differentiability at $\mu\in \mathscr{P}_{2}$)}
We say that $u$ is $L$-differentiable at $\mu\in \mathscr{P}_{2}$ if there is  an $X\in L^{2}(\Omega)$ such that $u = \mathscr{L}_{X}$ and the function $U: L^{2}(\Omega)\rightarrow \mathbb{R}$ given by $U(Y)\triangleq u(\mathscr{L}_{Y})$ is Fr$\acute{e}$chet differentiable at $X$. We will call U the lift of $u$.
\end{definition}

\begin{definition}{\bf (Lions derivative)}
If $u$ is $L$-differentiable at $\mu$ then we write $\partial_{\mu}u(\mathscr{L}_{X})(X):=DU(X)$, where $\partial_{\mu}u(\mathscr{L}_{X}):\mathbb{R}^{d}\rightarrow \mathbb{R}^{d}$, which is called Lions derivative of $u$ at $\mu=\mathscr{L}_{X}$. Moreover, we have $\partial_{\mu}u: \mathscr{P}_{2}\times \mathbb{R}^{d}\rightarrow  \mathbb{R}^{d}$ given by
$$\partial_{\mu}u(\mu,y):=[\partial_{\mu}u(\mu)](y).$$
\end{definition}
Moreover, $\partial_{\mu}u(\mu)\in L^{2}(\mu, \mathbb{R}^{d})$, for $\mu\in\mathscr{P}_{2}$. We observe that if $\mu$ is fixed then $\partial_{\mu}u(\mu)$ is a function from $\mathbb{R}^{d}\rightarrow
 \mathbb{R}^{d}$. Furthermore, if $\partial_{\mu}u(\mu):\mathbb{R}^{d}\rightarrow \mathbb{R}^{d}$ is differentiable at $y\in \mathbb{R}^{d}$, we denote its derivative by $\partial_{y}\partial_{\mu}u(\mu):\mathbb{R}^{d}\rightarrow \mathbb{R}^{d}\times \mathbb{R}^{d}$.

Let $\mid \cdot \mid$ be the Euclidean vector norm, $\langle\cdot,\cdot\rangle$ be the Euclidean inner product, $\parallel \cdot \parallel$ be the matrix norm or the operator norm if there is no confusion possible and $\parallel \cdot \parallel_{2}$ be the spectral norm.  We call a vector-valued, or matrix-valued function $u(\mu)~=~(u_{ij}(\mu))$ differentiable at $\mu\in\mathscr{P}_{2}$, if its all components are differentiable at $\mu$, and set $\partial
_{\mu}u(\mu):=(\partial_{\mu}u_{ij}(\mu))$. Furthermore, we call $\partial_{\mu}u(\mu)(y)$ differentiable at $y\in \mathbb{R}^{d}$, if all its
components are differentiable at $y$, and set $\partial_{y}\partial_{\mu}u(\mu)(y):=(\partial_{y}\partial_{\mu}u_{ij}(\mu)(y))$.

\begin{definition}
    We write $u\in C^{1,1}(\mathscr{P}_{2})$, if the map $u: \mathscr{P}_{2}\rightarrow \mathbb{R}$ is $L$-differentiable at any point $\mu\in\mathscr{P}_{2}$, and $\partial_{\mu}u: \mathscr{P}_{2}\times \mathbb{R}^{d}\rightarrow \mathbb{R}^{d}$ is Lipschitz continuous, that is, there is some real constant $C$ such that:
     \begin{eqnarray}
         |\partial_{\mu}u(\mu_{1},x_{1})-\partial_{\mu}u(\mu_{2},x_{2})|\leq C(\mathbb{W}(\mu_{1},\mu_{2})+|x_{1}-x_{2}|), ~\mu_{1},\mu_{2}\in\mathscr{P}_{2}, x_{1}, x_{2}\in\mathbb{R}^{d}.
    \end{eqnarray}
If moreover $\partial_{\mu}u(\mu,x)$ is bounded, we donote $u\in C_{b}^{1,1}(\mathscr{P}_{2})$, i.e.,
    \begin{eqnarray}
         |\partial_{\mu}u(\mu,x)|\leq C,\mu\in\mathscr{P}_{2}, x\in\mathbb{R}^{d}.
    \end{eqnarray}
\end{definition}

\begin{definition}
    We say that $\mu\in C^{1}(\mathbb{R}^{d})$ if the map $u: \mathbb{R}^{d}\rightarrow \mathbb{R}$ is differentiable at any point $x\in\mathbb{R}^{d}$ and there is some real constant $C$ such that
    \begin{eqnarray}
         |u(x_{1})-u(x_{2})|\leq C|x_{1}-x_{2}|, \quad x_{1}, x_{2}\in\mathbb{R}^{d}.
    \end{eqnarray}
    If moreover $\partial_{x}u(x)$ is bounded, we denote $u\in C_{b}^{1}(\mathbb{R}^{d})$, i.e.,
    \begin{eqnarray}
         |\partial_{x}u(x)|\leq C, \quad x\in\mathbb{R}^{d}.
    \end{eqnarray}
\end{definition}

\begin{definition}
    By $C_{b}^{1,(1,1)} (\mathbb{R}^{d}\times \mathscr{P}_{2})$ we denote the functions $u(x,\mu)$ such that $u(\cdot,\mu)\in C_{b}^{1} (\mathbb{R}^{d})$ for each $\mu$, and such that $u(x,\cdot)\in C_{b}^{(1,1)} (\mathscr{P}_{2})$ for each $x$.
\end{definition}

To conveniently express integrals with respect to the laws of process taken only over the "new" variable arising in the measure derivative we introduce another probability space $(\widetilde{\Omega},\widetilde{\mathcal{F}},\widetilde{\mathbb{P}})$ with a filtration $(\mathcal{F}_{t})_{t\geq 0}$ and processes $\widetilde{x}$ on this probability space such that they have the same laws as $x$. The expectation $\widetilde{\mathbb{E}}[\cdot] = \int_{\widetilde{\Omega}}(\cdot) {d}\widetilde{\mathbb{P}}$ acts only over the variables endowed with a tilde. Moreover, if we now consider the probability space $(\Omega\times\widetilde{\Omega},\mathcal{F}\times\widetilde{\mathcal{F}},\mathbb{P}\times\widetilde{\mathbb{P}})$$(=(\Omega\times\Omega,\mathcal{F}\times\mathcal{F},\mathbb{P}\times \mathbb{P}))$,  then we see that the processes with and without tilde are independent on this new space.

Next, we give out some assumptions throughout this work. We suppose that for any $T>0$, there exist constants $C_{T}$ such that the following conditions hold for all $t\in [0, T]$, $x, x_{1}, x_{2} \in \mathbb{R}^{d}$, $v, v_{1}, v_{2} \in \mathbb{R}^{d}$, $\mu,\mu_{1},\mu_{2}\in \mathscr{P}_{2}$.

\begin{assumption}\label{a1}
\begin{eqnarray}\label{e1}
&&\mid \bm{F}(\bm{x}_{1})-\bm{F}(\bm{x}_{2})\mid+\parallel \bm{\sigma}(\bm{x}_{1})-\bm{\sigma}(\bm{x}_{2})\parallel
\leq C_{T} \mid \bm{x}_{1}-\bm{x}_{2}\mid ;
\end{eqnarray}
\end{assumption}

\begin{assumption}\label{a2}
\begin{eqnarray}\label{e2}
&&\mid \bm{F}(\bm{x})\mid^{2}+\parallel \bm{\sigma}(\bm{x})\parallel^{2}
\leq C_{T}(1+\mid \bm{x}\mid^{2})  ;
\end{eqnarray}
\end{assumption}

\begin{assumption}\label{a3}
The function $\bm{\gamma}(\bm{x},\mu)\in C_{b}^{1,(1,1)}(\mathbb{R}^{d},\mathscr{P}_{2})$ and
\begin{eqnarray}
\parallel \bm{\gamma}(\bm{x}_{1},\mu_{1})-\bm{\gamma}(\bm{x}_{2},\mu_{2})\parallel\leq C_{T}\Big[\mid \bm{x}_{1}-\bm{x}_{2}\mid+\mathbb{W}(\mu_{1},\mu_{2})\Big],
\end{eqnarray}
\begin{eqnarray}\label{e3-1}
|\partial_{x}\bm{\gamma}_{ij}(\bm{x}_{1},\mu_{1})-\partial_{x}\bm{\gamma}_{ij}(\bm{x}_{2},\mu_{2})|\leq C_{T}\Big[\mid \bm{x}_{1}-\bm{x}_{2}\mid+\mathbb{W}(\mu_{1},\mu_{2})\Big].
\end{eqnarray}
\end{assumption}

\begin{assumption}\label{a4}
The smallest eigenvalue $\lambda_{1}(\bm{x},\mu)$ of the symmetric part $\frac{1}{2}(\bm{\gamma}+\bm{\gamma}^{*})$ of the matrix $\bm{\gamma}$ satisfies
\begin{eqnarray}\label{e4}
    \lambda_{1}(\bm{x},\mu)\geq C_{\lambda_{\bm{\gamma}}}>0.
\end{eqnarray}
\end{assumption}

\begin{remark}

Conditions (\ref{e1})-(\ref{e3-1}) are used to guarantee the existence and uniqueness
 of strong solutions to system  (\ref{maineq}). To prove tightness, we need the assumption \ref{a4} additionally.

\end{remark}

\section{Tightness and Well-posedness }
  \setcounter{equation}{0}
  \renewcommand{\theequation}
{3.\arabic{equation}}

\subsection{Tightness}

Firstly, we prove some uniform bounds $w.r.t.$ $\epsilon\in(0,1)$ for the 2nd moment of the solution $(\bm{x}_t^{\epsilon}, \sqrt{\epsilon}\bm{v}_{t}^{\epsilon})$ to system (\ref{maineq}).
\begin{lemma}\label{lem1}
    Assume that {assumptions \ref{a1}-\ref{a4}} hold. For every $T>0$, there exists a constant $C_{T}>0$ such that
    \begin{eqnarray}
        \mathbb {E}[\sup_{0\leq t\leq T}\mid \bm{x}_t^{\epsilon} \mid^{2}]\leq C_{T},
    \end{eqnarray}
    and for every $0\leq s\leq t\leq T$ and $\epsilon>0$, there exist a constant $C_{T}>0$ such that
    \begin{eqnarray}
        \mathbb {E}\mid \bm{x}_t^{\epsilon}-\bm{x}_{s}^{\epsilon}\mid^{2}\leq C_{T}|t-s|.
    \end{eqnarray}
\end{lemma}
\begin{proof}
First, we rewrite (\ref{maineq}) as
\begin{eqnarray}
      {d}\bm{v}_{t}^{\epsilon}+\frac{1}{\epsilon}\bm{\gamma}( \bm{x}_t^{\epsilon},\mathscr{L}_{\bm{x}_t^{\epsilon}}) \bm{v}_{t}^{\epsilon} {d}t = \frac{1}{\epsilon}F(\bm{x}_t^{\epsilon}) {d}t+\frac{1}{\epsilon}\bm{\sigma}(\bm{x}_t^{\epsilon}) {d}\bm{W}_{t}.
\end{eqnarray}
Let $\phi(t)$ satisfies
\begin{eqnarray}
    \frac{ {d}\phi(t)}{ {d}t} = -\frac{1}{\epsilon}\bm{\gamma}( \bm{x}_t^{\epsilon},\mathscr{L}_{\bm{x}_t^{\epsilon}})\phi(t), ~\phi(0) = I,
\end{eqnarray}
thus,
\begin{eqnarray}
     {d}(\phi^{-1}(t)\bm{v}_{t}^{\epsilon}) = \phi^{-1}(t) \frac{1}{\epsilon}F(\bm{x}_t^{\epsilon}) {d}t+\phi^{-1}(t) \frac{1}{\epsilon}\bm{\sigma}(\bm{x}_t^{\epsilon}) {d}\bm{W}_{t},
\end{eqnarray}
and hence
\begin{eqnarray}
   \bm{v}_{t}^{\epsilon} = \frac{1}{\epsilon}\int_{0}^{t}\phi(t)\phi^{-1}(s) F(\bm{x}_{s}^{\epsilon}) {d}s+\frac{1}{\epsilon}\int_{0}^{t}\phi(t)\phi^{-1}(s) \bm{\sigma}(\bm{x}_{s}^{\epsilon}) {d}\bm{W}_{s}+\phi(t)\bm{v}_{0}.
\end{eqnarray}
Therefore,
\begin{eqnarray}
    &&\bm{x}_{h}^{\epsilon}= \int_{0}^{h}\bm{v}_{t}^{\epsilon} {d}t + \bm{x}_{0}
    \nonumber\\
    &=& \frac{1}{\epsilon}\int_{0}^{h} \int_{0}^{t}\phi(t)\phi^{-1}(s) F(\bm{x}_{s}^{\epsilon}) {d}s  {d}t
    + \frac{1}{\epsilon}\int_{0}^{h}\int_{0}^{t}\phi(t)\phi^{-1}(s) \bm{\sigma}(\bm{x}_{s}^{\epsilon}) {d}\bm{W}_{s} {d}t
    +\int_{0}^{h}\phi(t)\bm{v}_{0} {d}t+ \bm{x}_{0}
    \nonumber\\
    &=&\sum_{i=1}^{3}I_{i}(h)+ \bm{x}_{0}.
\end{eqnarray}
Let $y(t)=\phi(t)\phi^{-1}(s) F(\bm{x}_{s}^{\epsilon})$, then $y(s)=F(\bm{x}_{s}^{\epsilon})$ and
\begin{eqnarray}
    \frac{ {d}y(t)}{ {d}t} &=& \frac{ {d}\phi(t)}{ {d}t}\phi^{-1}(s)F(\bm{x}_{s}^{\epsilon})
    \nonumber\\
    &=& -\frac{1}{\epsilon}\bm{\gamma}( \bm{x}_t^{\epsilon},\mathscr{L}_{\bm{x}_t^{\epsilon}})\phi(t)\phi^{-1}(s)F(\bm{x}_{s}^{\epsilon})
    \nonumber\\
    &=& -\frac{1}{\epsilon}\bm{\gamma}( \bm{x}_t^{\epsilon},\mathscr{L}_{\bm{x}_t^{\epsilon}})y(t),
\end{eqnarray}
thus, by Lemma 4.4.2 in \cite{HP:ordieq}, we have
\begin{eqnarray}\label{ppinvf}
    |\phi(t)\phi^{-1}(s) F(\bm{x}_{s}^{\epsilon})|&\leq&|F(\bm{x}_{s}^{\epsilon})|e^{-\int_{s}^{t}\frac{\lambda_{1}(\bm{x}_{r}^{\epsilon},\mathscr{L}_{\bm{x}_{r}^{\epsilon}})}{\epsilon} {d}r}
    \nonumber\\
    &\leq&|F(\bm{x}_{s}^{\epsilon})|e^{-\frac{C_{\lambda_{\bm{\gamma}}}}{\epsilon}(t-s)}.
    \end{eqnarray}
Similarly, we have
\begin{eqnarray}\label{sigeat}
    |\phi(t)\phi^{-1}(s) \bm{\sigma}(\bm{x}_{s}^{\epsilon})|
    &\leq&|\bm{\sigma}(\bm{x}_{s}^{\epsilon})|e^{-\frac{C_{\lambda_{\bm{\gamma}}}}{\epsilon}(t-s)},
\end{eqnarray}
and
\begin{eqnarray}\label{conest}
    |\phi(t)\bm{v}_{0}|&\leq&|\bm{v}_{0}|e^{-\frac{C_{\lambda_{\bm{\gamma}}}}{\epsilon}t}.
\end{eqnarray}
By Fubini Theorem, we have
\begin{eqnarray}\label{I1pre}
    |I_{1}(h)|&\leq&\frac{1}{\epsilon}\int_{0}^{h}\int_{s}^{h}|F(\bm{x}_{s}^{\epsilon})|e^{-\frac{C_{\lambda_{\bm{\gamma}}}}{\epsilon}(t-s)} {d}t {d}s
    \nonumber\\
    &=&\frac{1}{C_{\lambda_{\bm{\gamma}}}}\int_{0}^{h}|F(\bm{x}_{s}^{\epsilon})| {d}s,
\end{eqnarray}
and hence by \eqref{I1pre}, H\"{o}lder inequality and assumption \ref{a2}, we have
\begin{eqnarray}\label{est1}
     \mathbb {E}\sup_{0\leq h\leq T}|I_{1}(h)|^{2}
    &\leq&\frac{T^{2}}{C^{2}_{\lambda_{\bm{\gamma}}}}\mathbb {E}\int_{0}^{T}|F(\bm{x}_{s}^{\epsilon})|^{2} {d}s
    \nonumber\\
    &\leq&\frac{C_{T}}{C^{2}_{\lambda_{\bm{\gamma}}}}\int_{0}^{T}\mathbb {E}[1+|\bm{x}_{s}^{\epsilon}|^{2}] {d}s
    \nonumber\\
    &\leq&\frac{C_{T}}{C^{2}_{\lambda_{\bm{\gamma}}}} +\frac{C_{T}}{C^{2}_{\lambda_{\bm{\gamma}}}}\int_{0}^{T}\mathbb {E}\sup_{0\leq r\leq s}|\bm{x}_{r}^{\epsilon}|^{2} {d}s.
\end{eqnarray}
By Stochastics Fubini Theorem, we have
\begin{eqnarray}
    I_{2}(h) &=& \frac{1}{\epsilon}\int_{0}^{h}\int_{0}^{t}\phi(t)\phi^{-1}(s) \bm{\sigma}(\bm{x}_{s}^{\epsilon}) {d}\bm{W}_{s} {d}t
    \nonumber\\
    &=& \frac{1}{\epsilon}\int_{0}^{h}\int_{s}^{h}\phi(t)\phi^{-1}(s) \bm{\sigma}(\bm{x}_{s}^{\epsilon}) {d}t {d}\bm{W}_{s},
\end{eqnarray}
and let
\begin{eqnarray}
    \widehat{I_{2}} = \frac{1}{\epsilon}\int_{s}^{h}\phi(t)\phi^{-1}(s) \bm{\sigma}(\bm{x}_{s}^{\epsilon}) {d}t,
\end{eqnarray}
thus, from (\ref{sigeat}) we have
\begin{eqnarray}
    |\widehat{I_{2}}| &\leq& \frac{1}{\epsilon}\int_{s}^{h}|\phi(t)\phi^{-1}(s) \bm{\sigma}(\bm{x}_{s}^{\epsilon})| {d}t
    \nonumber\\
    &\leq& \frac{1}{\epsilon}\int_{s}^{h}|\bm{\sigma}(\bm{x}_{s}^{\epsilon})|e^{-\frac{C_{\lambda_{\bm{\gamma}}}}{\epsilon}(t-s)} {d}t
    \nonumber\\
    &\leq&|\bm{\sigma}(\bm{x}_{s}^{\epsilon})|\frac{1}{C_{\lambda_{\bm{\gamma}}}},
 \end{eqnarray}
    and hence
    \begin{eqnarray}\label{sigest2}
      |\widehat{I_{2}}|^{2} =  \Big|\frac{1}{\epsilon}\int_{s}^{h}\phi(t)\phi^{-1}(s) \bm{\sigma}(\bm{x}_{s}^{\epsilon}) {d}t\Big|^{2}\leq|\bm{\sigma}(\bm{x}_{s}^{\epsilon})|^{2}\frac{1}{C^{2}_{\lambda_{\bm{\gamma}}}}.
    \end{eqnarray}
By Doob's maximal inequality, (\ref{sigest2}), Burkholder inequality and assumption \ref{a2}, we have
\begin{eqnarray}\label{est2}
    \mathbb {E}\sup_{0\leq h\leq T}|I_{2}(h)|^{2}  &=&E\sup_{0\leq h\leq T}\Big|\int_{0}^{h}\Big(\frac{1}{\epsilon}\int_{s}^{h}\phi(t)\phi^{-1}(s) \bm{\sigma}(\bm{x}_{s}^{\epsilon}) {d}t\Big) {d}\bm{W}_{s}\Big|^{2}
     \nonumber\\
     &\leq&4 \mathbb {E}\Big|\int_{0}^{T}\Big(\frac{1}{\epsilon}\int_{s}^{h}\phi(t)\phi^{-1}(s) \bm{\sigma}(\bm{x}_{s}^{\epsilon}) {d}t\Big) {d}\bm{W}_{s}\Big|^{2}
     \nonumber\\
    &\leq&4\mathbb {E}\int_{0}^{T}|\bm{\sigma}(\bm{x}_{s}^{\epsilon})|^{2}\frac{1}{C^{2}_{\lambda_{\bm{\gamma}}}} {d}s
      \nonumber\\
     &\leq&\frac{C_{T}}{C^{2}_{\lambda_{\bm{\gamma}}}}\int_{0}^{T}\mathbb {E}[1 +|\bm{x}_{s}^{\epsilon}|^{2}] {d}s
    \nonumber\\
    &\leq& \frac{C_{T}}{C^{2}_{\lambda_{\bm{\gamma}}}} + \frac{C_{T}}{C^{2}_{\lambda_{\bm{\gamma}}}}\int_{0}^{T}\mathbb {E}\sup_{0\leq r\leq s}|\bm{x}_{r}^{\epsilon}|^{2} {d}s.
\end{eqnarray}
For $I_{3}(h)$, by Cauchy-Schwarz's inequality, H\"{o}lder inequality and (\ref{conest}), we have
\begin{eqnarray}\label{est3}
E\sup_{0\leq h\leq T}|I_{3}(h)|^{2}  &=&  \mathbb {E}\sup_{0\leq h\leq T}\Big|\int_{0}^{h}\phi(t)\bm{v}_{0} {d}t+ \bm{x}_{0}\Big|^{2}
\nonumber\\
&\leq&C_T(\mathbb {E}\sup_{0\leq h\leq T}\int_{0}^{h}\Big|\phi(t)\bm{v}_{0}\Big|^{2} {d}t+ \Big|\bm{x}_{0}\Big|^{2})
\nonumber\\
&\leq&C_T(\mathbb {E}\int_{0}^{T}\Big|\phi(t)\bm{v}_{0}\Big|^{2} {d}t+ \Big|\bm{x}_{0}\Big|^{2})
\nonumber\\
&\leq&C_T(\mathbb {E}\int_{0}^{T}e^{-\frac{2C_{\lambda_{\bm{\gamma}}}}{\epsilon}t}\Big|\bm{v}_{0}\Big|^{2} {d}t+ \Big|\bm{x}_{0}\Big|^{2})
\nonumber\\
&\leq&C_{T}(
\frac{\epsilon}{2C_{\lambda_{\bm{\gamma}}}}\Big|\bm{v}_{0}\Big|^{2}+\Big|\bm{x}_{0}\Big|^{2}).
\end{eqnarray}
Combining (\ref{est1}), (\ref{est2}) and (\ref{est3}), by Cauchy-Schwarz's inequality we have
\begin{eqnarray}
    \mathbb {E}\sup_{0\leq h\leq T}|\bm{x}_{h}^{\epsilon}|^{2} &\leq&C\left(\frac{C_{T}}{C^{2}_{\lambda_{\bm{\gamma}}}} +\frac{C_{T}}{C^{2}_{\lambda_{\bm{\gamma}}}}\int_{0}^{T}\mathbb {E}\sup_{0\leq r\leq s}|\bm{x}_{r}^{\epsilon}|^{2} {d}s+C_{T}(
\frac{\epsilon}{2C_{\lambda_{\bm{\gamma}}}}\Big|\bm{v}_{0}\Big|^{2}+\Big|\bm{x}_{0}\Big|^{2})\right)
    \nonumber\\
    &\leq&\frac{C_{T}}{C^{2}_{\lambda_{\bm{\gamma}}}}\int_{0}^{T}\mathbb {E}\sup_{0\leq r\leq s}|\bm{x}_{r}^{\epsilon}|^{2} {d}s+C_{T}(\frac{1}{C^{2}_{\lambda_{\bm{\gamma}}}} +  \frac{\epsilon}{2C_{\lambda_{\bm{\gamma}}}}\Big|\bm{v}_{0}\Big|^{2}+\Big|\bm{x}_{0}\Big|^{2}),
   \end{eqnarray}
thus, by Gronwall inequality we have
\begin{eqnarray}\label{xsqubound}
    \mathbb {E}\sup_{0\leq h\leq T}|\bm{x}_{h}^{\epsilon}|^{2}
    &\leq&C_{T}(\frac{1}{C^{2}_{\lambda_{\bm{\gamma}}}}+\frac{\epsilon}{2C_{\lambda_{\bm{\gamma}}}}\Big|\bm{v}_{0}\Big|^{2}+\Big|\bm{x}_{0}\Big|^{2})
    (1+\frac{C_{T}}{C^{2}_{\lambda_{\bm{\gamma}}}}e^{\frac{C_{T}}{C^{2}_{\lambda_{\bm{\gamma}}}}})\leq C_{T}.
\end{eqnarray}
For the proof of the second part, by Stochastic Fubini Theorem, we have
\begin{align}\label{xconpre}
\bm{x}_{h+k}-\bm{x}_{h}
     &=\frac{1}{\epsilon}\int_{0}^{h+k}\int_{s}^{h+k}\phi(t)\phi^{-1}(s) F(\bm{x}_{s}^{\epsilon}) {d}t  {d}s
   \nonumber\\
  &\ \ \ \ + \frac{1}{\epsilon}\int_{0}^{h+k}\int_{s}^{h+k} \phi(t)\phi^{-1}(s) \bm{\sigma}(\bm{x}_{s}^{\epsilon}) {d}t {d}\bm{W}_{s}
    +\int_{0}^{h+k}\phi(t)\bm{v}_{0} {d}t
    \nonumber\\
   &\ \ \ \ -\frac{1}{\epsilon}\int_{0}^{h}\int_{s}^{h}\phi(t)\phi^{-1}(s) F(\bm{x}_{s}^{\epsilon}) {d}t  {d}s
   \nonumber\\
    &\ \ \ \ - \frac{1}{\epsilon}\int_{0}^{h}\int_{s}^{h}\phi(t)\phi^{-1}(s) \bm{\sigma}(\bm{x}_{s}^{\epsilon}) {d}t {d}\bm{W}_{s}
    -\int_{0}^{h}\phi(t)\bm{v}_{0} {d}t
\nonumber\\
     &= \frac{1}{\epsilon}\Big(\int_{0}^{h}\int_{h}^{h+k}\phi(t)\phi^{-1}(s) F(\bm{x}_{s}^{\epsilon})  {d}t {d}s+ \int_{h}^{h+k}\int_{s}^{h+k}\phi(t)\phi^{-1}(s) F(\bm{x}_{s}^{\epsilon})  {d}t {d}s\Big)
    \nonumber\\
    &\ \ \ \ + \frac{1}{\epsilon}\Big(\int_{0}^{h}\int_{h}^{h+k} \phi(t)\phi^{-1}(s) \bm{\sigma}(\bm{x}_{s}^{\epsilon}) {d}t {d}\bm{W}_{s}
    + \int_{h}^{h+k}\int_{s}^{h+k} \phi(t)\phi^{-1}(s) \bm{\sigma}(\bm{x}_{s}^{\epsilon}) {d}t {d}\bm{W}_{s}\Big)
    \nonumber\\
    &\ \ \ \ +\int_{h}^{h+k}\phi(t)\bm{v}_{0} {d}t
    \nonumber\\
    &=\sum_{i=1}^{3}\widetilde{I}_{i}(h).
\end{align}
For $\widetilde{I}_{1}(h)$, by \eqref{ppinvf} and mean value Theorem we have
\begin{align}
    |\widetilde{I}_{1}(h)|&=\left|\frac{1}{\epsilon}\int_{0}^{h}\int_{h}^{h+k}\phi(t)\phi^{-1}(s) F(\bm{x}_{s}^{\epsilon})  {d}t {d}s+ \frac{1}{\epsilon}\int_{h}^{h+k}\int_{s}^{h+k}\phi(t)\phi^{-1}(s) F(\bm{x}_{s}^{\epsilon})  {d}t {d}s\right|
    \nonumber\\
    &\leq \left|\frac{1}{\epsilon}\int_{0}^{h}\int_{h}^{h+k}\phi(t)\phi^{-1}(s) F(\bm{x}_{s}^{\epsilon})  {d}t {d}s\right|+ \left|\frac{1}{\epsilon}\int_{h}^{h+k}\int_{s}^{h+k}\phi(t)\phi^{-1}(s) F(\bm{x}_{s}^{\epsilon})  {d}t {d}s\right|\nonumber\\
    &\leq\frac{1}{\epsilon}\int_{0}^{h}\int_{h}^{h+k}\left|\phi(t)\phi^{-1}(s) F(\bm{x}_{s}^{\epsilon})\right| {d}t {d}s+ \frac{1}{\epsilon}\int_{h}^{h+k}\int_{s}^{h+k}\left|\phi(t)\phi^{-1}(s) F(\bm{x}_{s}^{\epsilon})\right| {d}t {d}s\nonumber\\
    &\leq \frac{1}{\epsilon}\int_{0}^{h}\int_{h}^{h+k} |F(\bm{x}_{s}^{\epsilon})|e^{-\frac{C_{\lambda_{\bm{\gamma}}}}{\epsilon}(t-s)} {d}t {d}s+ \frac{1}{\epsilon}\int_{h}^{h+k}\int_{s}^{h+k}|F(\bm{x}_{s}^{\epsilon})|e^{-\frac{C_{\lambda_{\bm{\gamma}}}}{\epsilon}(t-s)} {d}t {d}s \nonumber\\
     &\leq \frac{1}{\epsilon}\int_{0}^{h} \frac{\epsilon}{C_{\lambda_{\bm{\gamma}}}}(e^{-\frac{C_{\lambda_{\bm{\gamma}}}}{\epsilon}(h-s)}-e^{-\frac{C_{\lambda_{\bm{\gamma}}}}{\epsilon}(h+k-s)})
     |F(\bm{x}_{s}^{\epsilon})|\mathrm{d}s\nonumber
     \end{align}
      \begin{align}\label{absI1con}
    &\ \ \ \ + \frac{1}{\epsilon}\int_{h}^{h+k}\frac{\epsilon}{C_{\lambda_{\bm{\gamma}}}}(1-e^{-\frac{C_{\lambda_{\bm{\gamma}}}}{\epsilon}(h+k-s)})|F(\bm{x}_{s}^{\epsilon})| {d}t {d}s\nonumber\\
    &\leq\frac{1}{C_{\lambda_{\bm{\gamma}}}}\int_{0}^{h} (\frac{C_{\lambda_{\bm{\gamma}}}}{\epsilon}e^{-\frac{C_{\lambda_{\bm{\gamma}}}}{\epsilon}\theta}k)|F(\bm{x}_{s}^{\epsilon})| {d}s + \frac{1}{C_{\lambda_{\bm{\gamma}}}}\int_{h}^{h+k}|F(\bm{x}_{s}^{\epsilon})| {d}t {d}s,
\end{align}
thus, by \eqref{absI1con},
H\"{o}lder inequality, assumption \ref{a2} and \eqref{xsqubound}, we get
\begin{align}\label{I1squcon}
    \mathbb {E}|\widetilde{I}_{1}(h)|^2&\leq \frac{2h}{C^2_{\lambda_{\bm{\gamma}}}}\mathbb {E}\int_{0}^{h} (\frac{C_{\lambda_{\bm{\gamma}}}}{\epsilon}e^{-\frac{C_{\lambda_{\bm{\gamma}}}}{\epsilon}\theta}k)^2|F(\bm{x}_{s}^{\epsilon})|^2 {d}s\nonumber\\
    &\ \ \ \ + \frac{2k}{C^2_{\lambda_{\bm{\gamma}}}}\mathbb {E}\int_{h}^{h+k}|F(\bm{x}_{s}^{\epsilon})|^2 {d}t {d}s\nonumber\\
    &\leq \frac{2hk^2}{C^2_{\lambda_{\bm{\gamma}}}}\int_{0}^{h} \mathbb {E}|\bm{x}_{s}^{\epsilon}|^2 {d}s + \frac{2k}{C^2_{\lambda_{\bm{\gamma}}}}\int_{h}^{h+k}\mathbb {E}|\bm{x}_{s}^{\epsilon}|^2 {d}t {d}s\nonumber\\
    &\leq \frac{2k^2}{C^2_{\lambda_{\bm{\gamma}}}}(h^2+1).
\end{align}
For $\widetilde{I}_{2}(h)$, let
\begin{align}
    \hat{I}_{21}(s) = \frac{1}{\epsilon}\int_{h}^{h+k} \phi(t)\phi^{-1}(s) \bm{\sigma}(\bm{x}_{s}^{\epsilon}) {d}t,
\end{align}
and
\begin{align}
    \hat{I}_{22}(s) = \frac{1}{\epsilon}\int_{s}^{h+k} \phi(t)\phi^{-1}(s) \bm{\sigma}(\bm{x}_{s}^{\epsilon}) {d}t,
\end{align}
similarly to \eqref{sigest2}, by \eqref{sigeat} and mean value Theorem we get
\begin{align}
    |\hat{I}_{21}(s)|&\leq \frac{1}{\epsilon}\int_{h}^{h+k}|\bm{\sigma}(\bm{x}_{s}^{\epsilon})|e^{-\frac{C_{\lambda_{\bm{\gamma}}}}{\epsilon}(t-s)} {d}t\nonumber\\
    &\leq \frac{1}{\epsilon}|\bm{\sigma}(\bm{x}_{s}^{\epsilon})|\left(\frac{\epsilon}{C_{\lambda_{\bm{\gamma}}}}(e^{-\frac{C_{\lambda_{\bm{\gamma}}}}{\epsilon}(h-s)}-e^{-\frac{C_{\lambda_{\bm{\gamma}}}}{\epsilon}(h+k-s)})\right)\nonumber\\
    &\leq\frac{1}{C_{\lambda_{\bm{\gamma}}}}|\bm{\sigma}(\bm{x}_{s}^{\epsilon})|\left(e^{-\frac{C_{\lambda_{\bm{\gamma}}}}{\epsilon}(h-s)}-e^{-\frac{C_{\lambda_{\bm{\gamma}}}}{\epsilon}(h+k-s)}\right)\nonumber\\
    &\leq \frac{1}{C_{\lambda_{\bm{\gamma}}}}|\bm{\sigma}(\bm{x}_{s}^{\epsilon})|\left(\frac{C_{\lambda_{\bm{\gamma}}}}{\epsilon}e^{-\frac{C_{\lambda_{\bm{\gamma}}}}{\epsilon}\theta}k\right),
\end{align}
and hence
\begin{align}\label{I21con}
    |\hat{I}_{21}(s)|^2 \leq  \frac{1}{C^2_{\lambda_{\bm{\gamma}}}}|\bm{\sigma}(\bm{x}_{s}^{\epsilon})|^2\left(e^{-\frac{C_{\lambda_{\bm{\gamma}}}}{\epsilon}(h-s)}-e^{-\frac{C_{\lambda_{\bm{\gamma}}}}{\epsilon}(h+k-s)}\right)^2.
\end{align}
Similarly,
\begin{align}\label{I22con}
     |\hat{I}_{22}(s)|^2 \leq  \frac{1}{C^2_{\lambda_{\bm{\gamma}}}}|\bm{\sigma}(\bm{x}_{s}^{\epsilon})|^2\left(1-e^{-\frac{C_{\lambda_{\bm{\gamma}}}}{\epsilon}(h+k-s)}\right)^2
     \leq  \frac{1}{C^2_{\lambda_{\bm{\gamma}}}}|\bm{\sigma}(\bm{x}_{s}^{\epsilon})|^2.
\end{align}
By \eqref{I21con}-\eqref{I22con}, Cauchy-Schwarz inequality, and Burkh\"{o}lder inequality
\begin{align}\label{I2squcon}
    \mathbb {E}\left|\widetilde{I}_{2}(h)\right|^2&\leq C(\mathbb{E}\left|\int_0^h \hat{I}_{21}(s) {d} W_s\right|^2+\mathbb {E}\left|\int_h^{h+k} \hat{I}_{22}(s) {d}W_s\right|^2)\nonumber\\
    &\leq C\left(\mathbb {E}\int_0^h \left|\hat{I}_{21}(s)\right|^2 {d}s+\mathbb {E}\int_h^{h+k} \left|\hat{I}_{22}(s)\right|^2 {d}s\right)\nonumber\\
    &\leq C\Big(\mathbb {E}\int_0^h \frac{1}{C^2_{\lambda_{\bm{\gamma}}}}|\bm{\sigma}(\bm{x}_{s}^{\epsilon})|^2\left(\frac{2C_{\lambda_{\bm{\gamma}}}}{\epsilon}e^{-\frac{2C_{\lambda_{\bm{\gamma}}}}{\epsilon}\theta}k\right)^2 {d}s +\mathbb {E}\int_h^{h+k} \frac{1}{C^2_{\lambda_{\bm{\gamma}}}}|\bm{\sigma}(\bm{x}_{s}^{\epsilon})|^2 {d}s\Big).
\end{align}
Hence by \eqref{I2squcon}, \eqref{xsqubound} and assumption \ref{a2}, we get
\begin{align}
    \mathbb {E}\left|\widetilde{I}_{2}(h)\right|^2&\leq \frac{C}{C^2_{\lambda_{\bm{\gamma}}}}\Big(\mathbb {E}\int_0^h k^2|\bm{x}_{s}^{\epsilon}|^2 {d}s +\mathbb {E}\int_h^{h+k} |\bm{x}_{s}^{\epsilon}|^2 {d}s\Big)\nonumber\\
    &\leq \frac{Ck}{C^2_{\lambda_{\bm{\gamma}}}}\Big(hk +1\Big).
\end{align}
For $\widetilde{I}_{3}(h)$, similarly to \eqref{est3}, by Cauchy-Schwarz's inequality,  H\"{o}lder inequality and (\ref{conest}), we have
\begin{eqnarray}\label{I3squcon}
\mathbb {E}|\widetilde{I}_{3}(h)|^{2}  &=&  \mathbb {E}\Big|\int_{h}^{h+k}\phi(t)\bm{v}_{0} {d}t\Big|^{2}
\nonumber\\
&\leq&k\mathbb {E}\int_{h}^{h+k}\Big|\phi(t)\bm{v}_{0}\Big|^{2} {d}t
\nonumber\\
&\leq&k\mathbb {E}\int_{h}^{h+k}e^{-\frac{2C_{\lambda_{\bm{\gamma}}}}{\epsilon}t}\Big|\bm{v}_{0}\Big|^{2} {d}t
\nonumber\\
&=&
\frac{k\epsilon}{2C_{\lambda_{\bm{\gamma}}}}\Big|\bm{v}_{0}\Big|^{2}\left(e^{-\frac{2C_{\lambda_{\bm{\gamma}}}}{\epsilon}h}-e^{-\frac{2C_{\lambda_{\bm{\gamma}}}}{\epsilon}(h+k)}\right)\nonumber\\
&\leq& \frac{k\epsilon}{2C_{\lambda_{\bm{\gamma}}}}\Big|\bm{v}_{0}\Big|^{2}\left(\frac{2C_{\lambda_{\bm{\gamma}}}}{\epsilon}e^{-\frac{2C_{\lambda_{\bm{\gamma}}}}{\epsilon}\theta}k\right)\nonumber\\
&\leq& \frac{k^2\epsilon}{2C_{\lambda_{\bm{\gamma}}}}\Big|\bm{v}_{0}\Big|^{2}.
\end{eqnarray}
Thus, by \eqref{xconpre},
\eqref{I1squcon}, \eqref{I2squcon} and \eqref{I3squcon}, we get
\begin{align}
 \mathbb {E}\left|\bm{x}_{h+k}-\bm{x}_h\right|^2\leq C(\sum_{i=1}^3\mathbb {E}|\widetilde{I}_{i}(h)|^{2})\leq Ck.\nonumber
\end{align}
\end{proof}

\begin{lemma}{\bf{(\cite{shamgv:smklim}, Lemma~3)}}\label{l2vcon}
 Assume that assumption \ref{a1}-\ref{a4}  hold.  For any fixed $t\in[0,T]$,
    \begin{eqnarray*}
        \mathbb {E}[\epsilon|\bm{v}_{t}^{\epsilon}|^{2}]\leq C,
    \end{eqnarray*}
    where $C$ is a constant independent of $\epsilon$ and of $t\leq T$.
\end{lemma}

\begin{remark}
Note that by Lemma \ref{lem1} $\bm{x}_t^{\epsilon}$ is tightness, let $\mathcal{K}:=\overline{\{\mathscr{L}_{\bm{x}_t^{\epsilon}}: \forall \epsilon>0, t\in[0,T]\}}$, by Prokhorov Theorem,  $\mathcal{K}$ is tight and relatively compact in $\mathscr{P}_{2}$.
\end{remark}

Now we prove a lemma which is useful in the proof of the limiting equation.
\begin{lemma}\label{gl2conall}
    Assume that assumption \ref{a1}-\ref{a2} hold, if $g(x,\mu,y):\mathbb{R}^{d}\times \mathscr{P}_{2}\times\mathbb{R}^{d}:\rightarrow\mathbb{R}$ such that $|g(x,\mu,y)|\leq C_{T}$,  for all $x\in K$
    \begin{eqnarray}\label{gvcon1}
        \lim_{\epsilon\rightarrow0}\mathbb {E}\Big(\sup_{0\leq t\leq T}\Big|\widetilde{\mathbb {E}}(\int_{0}^{t}g(\bm{x}_{s}^{\epsilon},\mathscr{L}_{\bm{x}_{s}^{\epsilon}},\widetilde{\bm{x}_{s}^{\epsilon}})\epsilon(\bm{v}_{s}^{\epsilon})_{i} {d}s)\Big|\Big)^{2} = 0,
    \end{eqnarray}
     \begin{eqnarray}\label{gvcon2}
        \lim_{\epsilon\rightarrow0}\mathbb {E}\Big(\sup_{0\leq t\leq T}\Big|\widetilde{\mathbb {E}}(\int_{0}^{t}g(\bm{x}_{s}^{\epsilon},\mathscr{L}_{\bm{x}_{s}^{\epsilon}},\widetilde{\bm{x}_{s}^{\epsilon}})\epsilon(\widetilde{\bm{v}}_{s}^{\epsilon})_{i} {d}s)\Big|\Big)^{2} = 0,
    \end{eqnarray}
     \begin{eqnarray}\label{gvcon3}
        \lim_{\epsilon\rightarrow0}\mathbb {E}\Big(\sup_{0\leq t\leq T}\Big|\widetilde{\mathbb {E}}(\int_{0}^{t}g(\bm{x}_{s}^{\epsilon},\mathscr{L}_{\bm{x}_{s}^{\epsilon}},\widetilde{\bm{x}_{s}^{\epsilon}})\epsilon(\widetilde{\bm{v}}_{s}^{\epsilon})_{i} {d}(\bm{W}_{s})_{j})\Big|\Big)^{2} = 0,
    \end{eqnarray}
       \begin{eqnarray}\label{gvcon4}
        \lim_{\epsilon\rightarrow0}\mathbb {E}\Big(\sup_{0\leq t\leq T}\Big|\widetilde{\mathbb {E}}(\int_{0}^{t}g(\bm{x}_{s}^{\epsilon},\mathscr{L}_{\bm{x}_{s}^{\epsilon}},\widetilde{\bm{x}_{s}^{\epsilon}})\epsilon(\bm{v}_{s}^{\epsilon})_{i} {d}(\widetilde{W}_{s})_{j})\Big|\Big)^{2} = 0.
    \end{eqnarray}
    \end{lemma}

\begin{proof}
First, by H\"{o}lder inequality and Lemma \ref{l2vcon} we have
 \begin{eqnarray}
      E\Big(\sup_{0\leq t\leq T}\Big|\widetilde{\mathbb {E}}(\int_{0}^{t}g(\bm{x}_{s}^{\epsilon},\mathscr{L}_{\bm{x}_{s}^{\epsilon}},\widetilde{\bm{x}_{s}^{\epsilon}})\epsilon(\bm{v}_{s}^{\epsilon})_{i} {d}s)\Big|\Big)^{2}
      &\leq& \mathbb {E}\Big(\widetilde{\mathbb {E}}(\sup_{0\leq t\leq T}\int_{0}^{t}\Big|g(\bm{x}_{s}^{\epsilon},\mathscr{L}_{\bm{x}_{s}^{\epsilon}},\widetilde{\bm{x}_{s}^{\epsilon}})\epsilon(\bm{v}_{s}^{\epsilon})_{i}\Big| {d}s)\Big)^{2}
      \nonumber\\
      &\leq&
      \mathbb {E}\Big(\widetilde{\mathbb{E}}(\int_{0}^{T}\Big|g(\bm{x}_{s}^{\epsilon},\mathscr{L}_{\bm{x}_{s}^{\epsilon}},\widetilde{\bm{x}_{s}^{\epsilon}})\epsilon(\bm{v}_{s}^{\epsilon})_{i}\Big| {d}s)\Big)^{2}
        \nonumber\\
        &\leq& T\mathbb {E}\Big(\widetilde{\mathbb {E}}(\int_{0}^{T}\Big|g(\bm{x}_{s}^{\epsilon},\mathscr{L}_{\bm{x}_{s}^{\epsilon}},\widetilde{\bm{x}_{s}^{\epsilon}})\epsilon(\bm{v}_{s}^{\epsilon})_{i}\Big|^{2} {d}s)\Big)
        \nonumber\\
        &\leq& C_{T}\widetilde{\mathbb {E}}\Big(\int_{0}^{T}\mathbb {E}\Big|\epsilon(\bm{v}_{s}^{\epsilon})_{i}\Big|^{2} {d}s\Big)
        \nonumber\\
        &\leq& C_{T}\epsilon.
    \end{eqnarray}
Therefore, we get (\ref{gvcon1}).
For (\ref{gvcon2}), which is proven similarly. By H\"{o}lder inequality and Lemma \ref{l2vcon} we have
 \begin{eqnarray}
        \mathbb {E}\Big(\sup_{0\leq t\leq T}\Big|\widetilde{\mathbb {E}}(\int_{0}^{t}g(\bm{x}_{s}^{\epsilon},\mathscr{L}_{\bm{x}_{s}^{\epsilon}},\widetilde{\bm{x}_{s}^{\epsilon}})\epsilon(\widetilde{\bm{v}}_{s}^{\epsilon})_{i} {d}s)\Big|\Big)^{2} &\leq& C_{T}\mathbb {E}\Big(\int_{0}^{T}\widetilde{\mathbb {E}}\Big|\epsilon(\widetilde{\bm{v}}_{s}^{\epsilon})_{i}\Big|^{2} {d}s\Big)
        \nonumber\\
        &\leq& C_{T}\epsilon,
    \end{eqnarray}
and hence (\ref{gvcon2}) is proven.

To estimate (\ref{gvcon3})-(\ref{gvcon4}), by It\^{o} isometry, we have
\begin{eqnarray}\label{g3pro1}
 \widetilde{\mathbb {E}}\Big(\int_{0}^{T}g(\bm{x}_{s}^{\epsilon},\mathscr{L}_{\bm{x}_{s}^{\epsilon}},\widetilde{\bm{x}_{s}^{\epsilon}})\epsilon(\bm{v}_{s}^{\epsilon})_{i} {d}(\widetilde{\mathbb {E}}_{s})_{j}\Big)^{2} &=& \widetilde{\mathbb {E}}\int_{0}^{T}\Big|g(\bm{x}_{s}^{\epsilon},\mathscr{L}_{\bm{x}_{s}^{\epsilon}},\widetilde{\bm{x}_{s}^{\epsilon}})\epsilon(\bm{v}_{s}^{\epsilon})_{i}\Big|^{2} {d}s
 \nonumber\\
 &\leq& C_{T}\int_{0}^{T}\widetilde{\mathbb {E}}\Big|\epsilon(\bm{v}_{s}^{\epsilon})_{i}\Big|^{2} {d}s.
\end{eqnarray}
By Doob's maximal inequality,
\begin{eqnarray}\label{g3pro2}
      \widetilde{\mathbb {E}}\Big(\sup_{0\leq t\leq T}\Big|\int_{0}^{t}g(\bm{x}_{s}^{\epsilon},\mathscr{L}_{\bm{x}_{s}^{\epsilon}},\widetilde{\bm{x}_{s}^{\epsilon}})\epsilon(\bm{v}_{s}^{\epsilon})_{i} {d}(\widetilde{W}_{s})_{j}\Big|\Big)^{2} &\leq& 4\widetilde{\mathbb {E}}\Big(\int_{0}^{T}g(\bm{x}_{s}^{\epsilon},\mathscr{L}_{\bm{x}_{s}^{\epsilon}},\widetilde{\bm{x}_{s}^{\epsilon}})\epsilon(\bm{v}_{s}^{\epsilon})_{i} {d}(\widetilde{W}_{s})_{j}\Big)^{2}
      \nonumber\\   &\leq&4C_{T}\int_{0}^{T}\widetilde{\mathbb {E}}\Big|\epsilon(\bm{v}_{s}^{\epsilon})_{i}\Big|^{2} {d}s.
\end{eqnarray}
Moreover, by H\"{o}lder inequality, (\ref{g3pro3}) and Lemma \ref{l2vcon} we have
 \begin{eqnarray}\label{g3pro3}
  &&\mathbb {E}\Big(\sup_{0\leq t\leq T}\Big|\widetilde{\mathbb {E}}(\int_{0}^{t}g(\bm{x}_{s}^{\epsilon},\mathscr{L}_{\bm{x}_{s}^{\epsilon}},\widetilde{\bm{x}_{s}^{\epsilon}})\epsilon(\bm{v}_{s}^{\epsilon})_{i} {d}(\widetilde{W}_{s})_{j})\Big|\Big)^{2}
  \nonumber\\
  &\leq&\mathbb {E}\Big(\widetilde{\mathbb {E}}\sup_{0\leq t\leq T}\Big|\int_{0}^{t}g(\bm{x}_{s}^{\epsilon},\mathscr{L}_{\bm{x}_{s}^{\epsilon}},\widetilde{\bm{x}_{s}^{\epsilon}})\epsilon(\bm{v}_{s}^{\epsilon})_{i} {d}(\widetilde{W}_{s})_{j}\Big|^{2}\Big)
  \nonumber\\
  &\leq& \mathbb {E}\Big(4C_{T}\int_{0}^{T}\widetilde{\mathbb {E}}\Big|\epsilon(\bm{v}_{s}^{\epsilon})_{i}\Big|^{2} {d}s\Big)
  \nonumber\\
  &=& 4C_{T}\widetilde{\mathbb {E}}\Big(\int_{0}^{T}\mathbb {E}\Big|\epsilon(\bm{v}_{s}^{\epsilon})_{i}\Big|^{2} {d}s\Big)
  \nonumber\\
  &\leq& C_{T}\epsilon,
\end{eqnarray}
and therefore (\ref{gvcon4}) is proven.

For (\ref{gvcon3}), by Doob's maximal inequality and It\^{o} isometry we have
\begin{eqnarray}\label{laspro}
    \mathbb {E}\Big(\sup_{0\leq t\leq T}\Big|\int_{0}^{t}g(\bm{x}_{s}^{\epsilon},\mathscr{L}_{\bm{x}_{s}^{\epsilon}},\widetilde{\bm{x}_{s}^{\epsilon}})\epsilon(\widetilde{\bm{v}}_{s}^{\epsilon})_{i} {d}(\bm{W}_{s})_{j}\Big|\Big)^{2}
&\leq&4C_{T}\int_{0}^{T}\mathbb {E}\Big|\epsilon(\widetilde{\bm{v}}_{s}^{\epsilon})_{i}\Big|^{2} {d}s
      \nonumber\\
      &=& 4C_{T}\int_{0}^{T}\Big|\epsilon(\widetilde{\bm{v}}_{s}^{\epsilon})_{i}\Big|^{2} {d}s.
\end{eqnarray}
Moreover, by H\"{o}lder inequality, (\ref{laspro}) and Lemma \ref{l2vcon} we have
\begin{eqnarray}
        &&\mathbb {E}\Big(\sup_{0\leq t\leq T}\Big|\widetilde{\mathbb {E}}(\int_{0}^{t}g(\bm{x}_{s}^{\epsilon},\mathscr{L}_{\bm{x}_{s}^{\epsilon}},\widetilde{\bm{x}_{s}^{\epsilon}})\epsilon(\widetilde{\bm{v}}_{s}^{\epsilon})_{i} {d}(\bm{W}_{s})_{j})\Big|\Big)^{2}
        \nonumber\\
        &\leq& \mathbb {E}\Big(\widetilde{\mathbb {E}}\sup_{0\leq t\leq T}\Big|\int_{0}^{t}g(\bm{x}_{s}^{\epsilon},\mathscr{L}_{\bm{x}_{s}^{\epsilon}},\widetilde{\bm{x}_{s}^{\epsilon}})\epsilon(\widetilde{\bm{v}}_{s}^{\epsilon})_{i} {d}(\bm{W}_{s})_{j}\Big|^{2}\Big)
  \nonumber\\
  &=&\widetilde{\mathbb {E}}\Big(\mathbb {E}\sup_{0\leq t\leq T}\Big|\int_{0}^{t}g(\bm{x}_{s}^{\epsilon},\mathscr{L}_{\bm{x}_{s}^{\epsilon}},\widetilde{\bm{x}_{s}^{\epsilon}})\epsilon(\widetilde{\bm{v}}_{s}^{\epsilon})_{i} {d}(\bm{W}_{s})_{j}\Big|^{2}\Big)
  \nonumber\\
  &\leq& \widetilde{\mathbb {E}}\Big(4C_{T}\int_{0}^{T}\mathbb {E}\Big|\epsilon(\widetilde{\bm{v}}_{s}^{\epsilon})_{i}\Big|^{2} {d}s\Big)
  \nonumber\\
  &\leq& C_{T}\epsilon,
    \end{eqnarray}
therefore (\ref{gvcon3}) is proven.
\end{proof}

\subsection{Well-posedness of equation (\ref{maineq}) }
In this section, we prove the existence and uniqueness of strong solutions to system (\ref{maineq}).
\begin{theorem}\label{thm1}
    Suppose that assumptions \ref{a1}-\ref{a3} hold. For any $\epsilon>0$, any given initial value $\bm{x}_0,\bm{v}_0 \in \mathbb{R}^{d}$ there exists a unique solution ${(\bm{x}_t^{\epsilon},\bm{v}_{t}^{\epsilon}),t>0}$ for system (\ref{maineq}) and
    \begin{eqnarray}
\left\{\begin{array}{ll}
 \bm{x}_t^{\epsilon} = \bm{x}_{0}+\int_{0}^{t}\bm{v}_{s}^{\epsilon} {d}s,\\
 \bm{v}_{t}^{\epsilon} = \bm{v}_{0}+\frac{1}{\epsilon}\int_{0}^{t}F(\bm{x}_{s}^{\epsilon}) {d}s-\frac{1}{\epsilon}\int_{0}^{t}\bm{\gamma}( \bm{x}_{s}^{\epsilon},\mathscr{L}_{\bm{x}_{s}^{\epsilon}}) \bm{v}_{s}^{\epsilon} {d}s+\frac{1}{\epsilon}\int_{0}^{t}\bm{\sigma}(\bm{x}_{s}^{\epsilon}) {d}\bm{W}_{s}.
\end{array}
\right.
\end{eqnarray}

\end{theorem}

\begin{proof}
Without loss of generality, let $\epsilon=1$.  Let $\bm{z}_t =(\bm{x}_t,  \bm{v}_t)$,  $\tilde{A}(\bm{z}_t)= (\bm{v}_t, F(\bm{x}_t) )$,
$\tilde{B}(\bm{z}_t,  \mathscr{L}_{\bm{z}_{t}})= (0, \bm{\gamma}(\bm{x}_t, \mathscr{L}_{{\bm{x}}_{t}} )\bm{v}_t)$, $\tilde{\sigma}(\bm{z}_t)= (0,  \bm{\sigma}(\bm{x}_t))$.
The equation~(\ref{maineq}) is equivalent to the following equation
\begin{eqnarray}\label{fixed-1}
{d}\bm{z}_t= \tilde{A}(\bm{z}_t ){d}t + \tilde{B}(\bm{z}_t,  \mathscr{L}_{\bm{z}_{t}}){d}t +\tilde{ {\sigma}}(\bm{z}_t) {d}W_t.
\end{eqnarray}
By assumptions \ref{a1}-\ref{a3}, the nonlinearity  $\tilde{B}(\bm{z}_t,  \mathscr{L}_{\bm{z}_{t}})$ is locally Lipschitz, that is for every $R>0$\,,  $\rho_1, \rho_2 \in  \mathscr{P}_{2}$ and $\bm{z}_{1}\,, \bm{z}_{2}\in  \mathbb{R}^{2d}$
\begin{eqnarray}\label{assump-1}
| \tilde{B}(z_1,\rho_1 ) -\tilde{B}(z_2,\rho_2) |\leq  L_{ R}[ W_1(\rho_1, \rho_2 )  +  |z_1-z_2|] ,
\end{eqnarray}
for some constant $L_{ R}$  with $|z_{1}|+|z_{2}|\leq R$\,.
Define the random stopping time $\tau(R)$ by
$$\tau(R) = \inf \{t>0: |\bm{z}_t | \geq R   \},    $$
and denote by $\chi_I$ the characteristic function of the set $I$.
Now, for $t \leq T$, we define
$$G(\bm{z}_t) :=  \bm{z}_0 + \int_0^t \tilde{A}(\bm{z}_s)\chi_{\{s< \tau(R) \}} {d}s + \int_0^t \tilde{B}(\bm{z}_t,  \mathscr{L}_{\bm{z}_{s}})(z_s)\chi_{\{s< \tau(R) \}} {d}s +\int_0^t \tilde{\sigma}(\bm{z}_s) \chi_{\{s< \tau(R) \}}{d}W_s, $$
for any $\bm{z}_1=(\bm{x}_1, \bm{v}_1)$, $\bm{z}_2=(\bm{x}_2, \bm{v}_2)$ and $\bm{z}_1|_{t=0}= \bm{z}_2|_{t=0} = \bm{z}_0$,  using assumption \ref{a1}, then
\begin{eqnarray}
 | \int_0^t(\tilde{A}(\bm{z}_1 )- \tilde{A}(\bm{z}_2))\chi_{\{s< \tau(R) \}}  {d}s |^2
&\leq & T \int_0^t |\bm{v}_1- \bm{v}_2|^2+  |F(\bm{x}_1)- F(\bm{x}_2)|^2 {d}s  \nonumber\\
&\leq & T\int_0^t |\bm{v}_1- \bm{v}_2|^2+ C^2 |\bm{x}_1 - \bm{x}_2 |^2 {d}s  \nonumber\\
&\leq& C_T \int_0^t | \bm{z}_1 -  \bm{z}_2 |^2 {d}s .
\end{eqnarray}
Using (\ref{assump-1}), then
\begin{eqnarray}
|\int_0^t[ \tilde{B}(\bm{z}_1,  \mathscr{L}_{\bm{z}_{1}})(\bm{z}_1) - \tilde{B}(\bm{z}_2,  \mathscr{L}_{\bm{z}_{2}})(\bm{z}_1) ] \chi_{\{s< \tau(R) \}} {d}s |^2
  &\leq&  L_{R}T\int_0^t |\bm{z}_1- \bm{z}_2|^2 + \mathbb{W}^2( \mathscr{L}_{{\bm{x}}_{1}},  \mathscr{L}_{{\bm{x}}_{2}}) {d}s  \nonumber\\
 &\leq& C_{R,T }\int_0^t | \bm{z}_1 -  \bm{z}_2 |^2 {d}s .
\end{eqnarray}
By the Cauchy-Schwarz inequality,
\begin{eqnarray}
 \mathbb{E}\sup_{t\in [0,T]}  |\int_0^t ( \tilde{ {\sigma}}(\bm{z}_1  )- \tilde{  {\sigma}}(\bm{z}_2  )) \chi_{\{s< \tau(R) \}} {d}W_s |^2
&\leq &  \mathbb{E} ( \int_0^T |\tilde{ {\sigma}}(\bm{z}_1 ) - \tilde{ {\sigma}}(\bm{z}_2 )| {d}s)^2   \nonumber\\
&\leq & T C^2  \int_0^t\mathbb {E}[| \bm{\sigma}(\bm{x}_1 )- \bm{\sigma}(\bm{x}_2 ) |^{2} ]  {d}s ]\nonumber\\
&\leq & C_T   \int_0^t\mathbb {E} | \bm{z}_1 - \bm{z}_2  |^{2}   {d}s  .
\end{eqnarray}
Hence, the mapping $G$ is Lipschiz continuous in $L^2(\Omega, \mathbb{R}^{2d} )$, and the Lipschitz constant is strictly less than 1 in the interval $[0, T_0]$ for some $T=T_0$ sufficiently small. As the same arguments can be repeated in the $[T_0, 2T_0]$, $[2T_0, 3T_0]$ and so on.
 We obtain that system~(\ref{fixed-1}) has  a unique solution $ \bm{z}_{t}$ in $[0,T]$   for $t< \tau(R)\wedge T $.  As follows from
 Lemma~\ref{lem1} and  Lemma~\ref{l2vcon}, $\tau(R) \rightarrow  \infty $ as $R \rightarrow \infty$. The proof is complete.
\end{proof}

Combining Lemma \ref{lem1} with Theorem \ref{thm1} and assumptions \ref{a1}, \ref{a4}, Lemma 2 in \cite{shamgv:smklim} is true and we restate it for our purposes.
\begin{lemma}{\bf{(\cite{shamgv:smklim}, Lemma~2)}}\label{pconver}
    Suppose that assumptions \ref{a1}-\ref{a4} hold, then for any $p\geq 1$, $\epsilon \bm{v}_{t}^{\epsilon}\rightarrow0$ as $\epsilon\rightarrow0$ in $L^{p}$ with respect to $C_{\mathbb{R}^{d}}[0,T]$, and hence also in probability with respect to  $C_{\mathbb{R}^{d}}[0,T]$, i.e.,
    \begin{eqnarray}\label{lpconv}
        \lim_{\epsilon\rightarrow0}\mathbb {E}\Big[\Big(\sup_{0\leq t\leq T}\mid\epsilon \bm{v}_{t}^{\epsilon}\mid\Big)^{p}\Big] = 0,
    \end{eqnarray}
    and for all $\epsilon>0$,
    \begin{eqnarray}\label{proconv}
        \lim_{\epsilon\rightarrow0}\mathbb{P}\Big(\sup_{0\leq t\leq T}\mid\epsilon \bm{v}_{t}^{\epsilon}\mid>0\Big) = 0.
    \end{eqnarray}
Moreover, if $p=4$, then
        \begin{eqnarray}\label{fourconv}
        \mathbb {E}\Big[\Big(\sup_{0\leq t\leq T}\mid\epsilon \bm{v}_{t}^{\epsilon}\mid\Big)^{4}\Big] \leq C\epsilon.
    \end{eqnarray}
\end{lemma}
\begin{proof}
     Since conditions of Lemma 2 in \cite{shamgv:smklim} are satisfied from Lemma \ref{lem1}, Theorem \ref{thm1} and assumptions \ref{a1} and \ref{a4}, (\ref{lpconv}) and (\ref{proconv}) is immediately derived.
     Now, we turn to (\ref{fourconv}). Taking $q=2$ in Lemma 2 in \cite{shamgv:smklim}, we get
     \begin{eqnarray}
        \mathbb {E}\Big[\Big(\sup_{0\leq t\leq T}\mid\epsilon \bm{v}_{t}^{\epsilon}\mid^{2}\Big)^{2}\Big] &=& \int_{0}^{\infty} 2x \mathbb{P}\Big(\sup_{0\leq t\leq T}\mid\epsilon \bm{v}_{t}^{\epsilon}\mid^{2}>x\Big) {d}x
        \nonumber\\
        &\leq& \int_{0}^{\infty} 4\bm{\gamma} x \frac{2}{x} e^{-\frac{3C_{\lambda_{\bm{\gamma}}}}{32\epsilon\bm{\gamma}}} {d}x
        \nonumber\\
        &\leq& C\epsilon,
        \nonumber
    \end{eqnarray}
    where $\bm{\gamma} = C_{T}^{2}d$.
\end{proof}

\section{The rate of convergence}
  \setcounter{equation}{0}
  \renewcommand{\theequation}
{4.\arabic{equation}}
The following Theorem implies that the component $\bm{x}_t^{\epsilon}$ strongly converges to the solution $\bm{x}_{t}$ with convergence order $\sqrt{\epsilon}$.
\begin{theorem}\label{mainthm}
    Assume that assumptions \ref{a1}-\ref{a4} hold.  $\bm{x}_t^{\epsilon}$ is the solution of the equation~(\ref{maineq}),   for any $T>0$, and $\bm{x}_0,\bm{v}_0\in \mathbb{R}^{d}$,
    \begin{eqnarray}
        \mathbb{E}(\sup_{0\leq t\leq T}\mid \bm{x}_t^{\epsilon}-\bm{x}_{t}\mid^2)\leq C\sqrt{\epsilon},
    \end{eqnarray}
    where $C$ is a constant depending $T$ , $\bm{x}_0$ and $\bm{v}_0$.
     Here, $\bm{x}_{t}$ is the solution of the limiting equation~(\ref{maineq-limit})
\end{theorem}

\begin{proof}
Firstly, we show that
\begin{eqnarray}
    \bm{x}_t^{\epsilon} = \bm{x}_t^{\epsilon,1}+\bm{x}_t^{\epsilon,2},
\end{eqnarray}
where
\begin{eqnarray}\label{xe1}
    (\bm{x}_t^{\epsilon,1})_{i} &=& (\bm{x}_{0})_{i}+\int_{0}^{t}\bm{\gamma}_{ij}^{-1}( \bm{x}_{s}^{\epsilon},\mathscr{L}_{\bm{x}_{s}^{\epsilon}})F_{j}(\bm{x}_{s}^{\epsilon}) {d}s+\int_{0}^{t}\bm{\gamma}_{ij}^{-1}( \bm{x}_{s}^{\epsilon},\mathscr{L}_{\bm{x}_{s}^{\epsilon}})\bm{\sigma}_{jk}(\bm{x}_{s}^{\epsilon}) {d}(\bm{W}_{s})_{k}
    \nonumber\\
    &&- \bm{\gamma}_{ij}^{-1}( \bm{x}_t^{\epsilon},\mathscr{L}_{\bm{x}_t^{\epsilon}})\epsilon (\bm{v}_{t}^{\epsilon})_{j}+\bm{\gamma}_{ij}^{-1}( \bm{x}_{0},\mathscr{L}_{\bm{x}_{0}})\epsilon (\bm{v}_{0})_{j}
    \nonumber\\
    &&+(U_{t}^{\epsilon})_{i}+\int_{0}^{t} [\partial_{x_{l}}\bm{\gamma}_{ij}^{-1}( \bm{x}_{s}^{\epsilon},\mathscr{L}_{\bm{x}_{s}^{\epsilon}})]J_{jl}(\bm{x}_{s}^{\epsilon},\mathscr{L}_{\bm{x}_{s}^{\epsilon}}) {d}s+\int_{0}^{t} [\partial_{x_{l}}\bm{\gamma}_{ij}^{-1}( \bm{x}_{s}^{\epsilon},\mathscr{L}_{\bm{x}_{s}^{\epsilon}})]\times
    \nonumber\\
    &&~~~~~~~~\Big[-\int_{0}^{\infty}(e^{-\bm{\gamma}( \bm{x}_{s}^{\epsilon},\mathscr{L}_{\bm{x}_{s}^{\epsilon}})y})_{jk_{1}}(e^{-\bm{\gamma}^{*}( \bm{x}_{s}^{\epsilon},\mathscr{L}_{\bm{x}_{s}^{\epsilon}})y})_{k_{2}l} {d}y\Big] {d}[(\epsilon \bm{v}_{s}^{\epsilon})_{k_{1}}(\epsilon \bm{v}_{s}^{\epsilon})_{k_{2}}]
    \nonumber\\
    &&+\int_{0}^{t}\widetilde{\mathbb {E}}\Big[ (\partial_{\mu}\bm{\gamma}_{ij}^{-1}( \bm{x}_{s}^{\epsilon},\mathscr{L}_{\bm{x}_{s}^{\epsilon}})(\widetilde{\bm{x}_{s}^{\epsilon}}))_{l}\widetilde{J}_{jl}
    (\bm{x}_{s}^{\epsilon},\widetilde{\bm{x}_{s}^{\epsilon}},\mathscr{L}_{\bm{x}_{s}^{\epsilon}})\Big] {d}s,
\end{eqnarray}
and
\begin{eqnarray}\label{xe2}
    (\bm{x}_t^{\epsilon,2})_{i} &=&  \widetilde{\mathbb {E}}\Big[\int_{0}^{t}(\partial_{\mu}\bm{\gamma}_{ij}^{-1}( \bm{x}_{s}^{\epsilon},\mathscr{L}_{\bm{x}_{s}^{\epsilon}})(\widetilde{\bm{x}_{s}^{\epsilon}}))_{l}\times
    \nonumber\\
    &&~~~~~~~~\Big(-\int_{0}^{\infty} (e^{-\bm{\gamma}( \bm{x}_t^{\epsilon},\mathscr{L}_{\bm{x}_t^{\epsilon}})y})_{jk_{1}}(e^{-\bm{\gamma}^{*}( \widetilde{\bm{x}_t^{\epsilon}},\mathscr{L}_{\bm{x}_t^{\epsilon}})y})_{k_{2}l} {d}y {d}[(\epsilon \bm{v}_{t}^{\epsilon})_{k_{1}}(\epsilon \widetilde{\bm{v}_{t}^{\epsilon}})^{*}_{k_{2}}]\Big)\Big]
    \nonumber\\
    &&+ \widetilde{\mathbb {E}}\Big[\int_{0}^{t}(\partial_{\mu}\bm{\gamma}_{ij}^{-1}( \bm{x}_{s}^{\epsilon},\mathscr{L}_{\bm{x}_{s}^{\epsilon}})(\widetilde{\bm{x}_{s}^{\epsilon}}))_{l}\times\Big(-\int_{0}^{\infty} (e^{-\bm{\gamma}( \bm{x}_t^{\epsilon},\mathscr{L}_{\bm{x}_t^{\epsilon}})y})_{jk_{1}} (e^{-\bm{\gamma}^{*}( \widetilde{\bm{x}_t^{\epsilon}},\mathscr{L}_{\bm{x}_t^{\epsilon}})y})_{k_{2}l} {d}y
    \nonumber\\
    &&~~~~~~~~~~~~~~~~\times[(\epsilon \bm{v}_{t}^{\epsilon}F^{*}(\widetilde{\bm{x}_t^{\epsilon}}))_{k_{1}k_{2}} {d}t+(\epsilon \bm{v}_{t}^{\epsilon}(\bm{\sigma}(\widetilde{\bm{x}_t^{\epsilon}}) {d}\widetilde{\bm{W}_{t}})^{*})_{k_{1}k_{2}}
    \nonumber\\
    &&~~~~~~~~~~~~~~~~~~~~~+ (F(\bm{x}_t^{\epsilon})(\epsilon \widetilde{\bm{v}_{t}^{\epsilon}})^{*})_{k_{1}k_{2}} {d}t+(\bm{\sigma}(\bm{x}_t^{\epsilon}) {d}\bm{W}_{t}(\epsilon \widetilde{\bm{v}_{t}^{\epsilon}})^{*})_{k_{1}k_{2}}]\Big)\Big].
\end{eqnarray}
In fact, to determine the limit of  equation(\ref{maineq}) as $\epsilon\rightarrow0$, we rewrite the equation for $\bm{v}_{t}^{\epsilon}$ as
    $$\bm{\gamma}( \bm{x}_t^{\epsilon},\mathscr{L}_{\bm{x}_t^{\epsilon}}) \bm{v}_{t}^{\epsilon} {d}t = F(\bm{x}_t^{\epsilon}) {d}t+\bm{\sigma}(\bm{x}_t^{\epsilon}) {d}\bm{W}_{t}-\epsilon  {d}\bm{v}_{t}^{\epsilon},$$
and hence rewrite the equation for $\bm{x}_t^{\epsilon}$ as
 \begin{eqnarray}\label{xinteq}
        \bm{x}_t^{\epsilon} &=& \bm{x}_{0}+\int_{0}^{t}\bm{\gamma}^{-1}( \bm{x}_{s}^{\epsilon},\mathscr{L}_{\bm{x}_{s}^{\epsilon}})F(\bm{x}_{s}^{\epsilon}) {d}s+\int_{0}^{t}\bm{\gamma}^{-1}( \bm{x}_{s}^{\epsilon},\mathscr{L}_{\bm{x}_{s}^{\epsilon}})\bm{\sigma}(\bm{x}_{s}^{\epsilon}) {d}\bm{W}_{s}
        \nonumber\\
        &-&\epsilon \int_{0}^{t}\bm{\gamma}^{-1}( \bm{x}_{s}^{\epsilon},\mathscr{L}_{\bm{x}_{s}^{\epsilon}}) {d}\bm{v}_{s}^{\epsilon}.
    \end{eqnarray}
To determine the limit of the expression (\ref{xinteq}) as $\epsilon\rightarrow0$, we consider its ith component. Integrating by parts the last term on the right-hand side of equation (\ref{xinteq}), we obtain
\begin{eqnarray}\label{intpareq}
    \epsilon \int_{0}^{t}\bm{\gamma}_{ij}^{-1}( \bm{x}_{s}^{\epsilon},\mathscr{L}_{\bm{x}_{s}^{\epsilon}}) {d}(\bm{v}_{s}^{\epsilon})_{j}
    &=& \bm{\gamma}_{ij}^{-1}( \bm{x}_t^{\epsilon},\mathscr{L}_{\bm{x}_t^{\epsilon}})\epsilon (\bm{v}_{t}^{\epsilon})_{j}-\bm{\gamma}_{ij}^{-1}( \bm{x}_{0},\mathscr{L}_{\bm{x}_{0}})\epsilon (\bm{v}_{0})_{j}
    \nonumber\\
    &-&\int_{0}^{t} \epsilon(\bm{v}_{s}^{\epsilon})_{j} {d}\bm{\gamma}_{ij}^{-1}( \bm{x}_{s}^{\epsilon},\mathscr{L}_{\bm{x}_{s}^{\epsilon}}).
\end{eqnarray}
Using It\^{o} formula for functions of measures (see, \cite{BLPR:mestdi}, \cite{HvS:mcsdme}), the last term can be rewritten as
\begin{eqnarray}\label{itomeaeq}
    \int_{0}^{t} \epsilon(\bm{v}_{s}^{\epsilon})_{j} {d}\bm{\gamma}_{ij}^{-1}( \bm{x}_{s}^{\epsilon},\mathscr{L}_{\bm{x}_{s}^{\epsilon}}) &=& \int_{0}^{t} \epsilon(\bm{v}_{s}^{\epsilon})_{j}(\bm{v}_{s}^{\epsilon})_{l}\partial_{x_{l}}\bm{\gamma}_{ij}^{-1}( \bm{x}_{s}^{\epsilon},\mathscr{L}_{\bm{x}_{s}^{\epsilon}}) {d}s
    \nonumber\\
    &+&\int_{0}^{t} \epsilon(\bm{v}_{s}^{\epsilon})_{j}\widetilde{\mathbb {E}}[(\widetilde{\bm{v}_{s}^{\epsilon}})_{l}(\partial_{\mu}\bm{\gamma}_{ij}^{-1}( \bm{x}_{s}^{\epsilon},\mathscr{L}_{\bm{x}_{s}^{\epsilon}})(\widetilde{\bm{x}_{s}^{\epsilon}}))_{l}] {d}s
    \nonumber\\
    &\triangleq& I_{1}+I_{2}.
\end{eqnarray}
Note that $\bm{x}_t^{\epsilon}$ has bounded variation, hence the It\^{o}   term in the above equation is zero.

 For the first term $I_{1}$, we use the result in \cite{shamgv:smklim} to get the following equality immediately
\begin{eqnarray}\label{i1laeq}
    I_{1} &=& (U_{t}^{\epsilon})_{i}+\int_{0}^{t} [\partial_{x_{l}}\bm{\gamma}_{ij}^{-1}( \bm{x}_{s}^{\epsilon},\mathscr{L}_{\bm{x}_{s}^{\epsilon}})]J_{jl}(\bm{x}_{s}^{\epsilon},\mathscr{L}_{\bm{x}_{s}^{\epsilon}}) {d}s+\int_{0}^{t} [\partial_{x_{l}}\bm{\gamma}_{ij}^{-1}( \bm{x}_{s}^{\epsilon},\mathscr{L}_{\bm{x}_{s}^{\epsilon}})]\times
    \nonumber\\
    &&\Big[-\int_{0}^{\infty}(e^{-\bm{\gamma}( \bm{x}_{s}^{\epsilon},\mathscr{L}_{\bm{x}_{s}^{\epsilon}})y})_{jk_{1}}(e^{-\bm{\gamma}^{*}( \bm{x}_{s}^{\epsilon},\mathscr{L}_{\bm{x}_{s}^{\epsilon}})y})_{k_{2}l} {d}y\Big] {d}[(\epsilon \bm{v}_{s}^{\epsilon})_{k_{1}}(\epsilon \bm{v}_{s}^{\epsilon})_{k_{2}}],
\end{eqnarray}
where
\begin{eqnarray}
    (U_{t}^{\epsilon})_{i} &=& 
    \int_{0}^{t} [\partial_{x_{l}}\bm{\gamma}_{ij}^{-1}( \bm{x}_{s}^{\epsilon},\mathscr{L}_{\bm{x}_{s}^{\epsilon}})]\times \Big[\int_{0}^{\infty}(e^{-\bm{\gamma}( \bm{x}_{s}^{\epsilon},\mathscr{L}_{\bm{x}_{s}^{\epsilon}})y})_{jk_{1}}(e^{-\bm{\gamma}^{*}( \bm{x}_{s}^{\epsilon},\mathscr{L}_{\bm{x}_{s}^{\epsilon}})y})_{k_{2}l} {d}y \times
    \nonumber\\
    &&\Big((\epsilon \bm{v}_{s}^{\epsilon}F^{*}(\bm{x}_{s}^{\epsilon}))_{k_{1}k_{2}} {d}s+(\epsilon \bm{v}_{s}^{\epsilon}(\bm{\sigma}(\bm{x}_{s}^{\epsilon}) {d}\bm{W}_{s})^{*})_{k_{1}k_{2}}+
    \nonumber\\
    &&(F(\bm{x}_{s}^{\epsilon})(\epsilon \bm{v}_{s}^{\epsilon})^{*})_{k_{1}k_{2}}+(\bm{\sigma}(\bm{x}_{s}^{\epsilon}) {d}\bm{W}_{s}(\epsilon \bm{v}_{s}^{\epsilon})^{*})_{k_{1}k_{2}}\Big)\Big],
    \nonumber
  \end{eqnarray}
  and
 \begin{eqnarray}\label{thm-new-1}
     J(\bm{x}_t^{\epsilon},\mathscr{L}_{\bm{x}_t^{\epsilon}}) &=& \int_{0}^{\infty} e^{-\bm{\gamma}( \bm{x}_t^{\epsilon},\mathscr{L}_{\bm{x}_t^{\epsilon}})y}(\bm{\sigma}(\bm{x}_t^{\epsilon})\bm{\sigma}^{*}(\bm{x}_t^{\epsilon}))e^{-\bm{\gamma}^{*}( \bm{x}_t^{\epsilon},\mathscr{L}_{\bm{x}_t^{\epsilon}})y} {d}y,
\end{eqnarray}
$J(\bm{x}_t^{\epsilon},\mathscr{L}_{\bm{x}_t^{\epsilon}}):\mathbb{R}^{d}\times \mathscr{P}_{2}\rightarrow\mathbb{R}^{d}$ is the solution to the Lyapunov equation
\begin{eqnarray}
\bm{\gamma}(\bm{x}_t^{\epsilon},\mathscr{L}_{\bm{x}_t^{\epsilon}})J(\bm{x}_t^{\epsilon},\mathscr{L}_{\bm{x}_t^{\epsilon}})+J(\bm{x}_t^{\epsilon},\mathscr{L}_{\bm{x}_t^{\epsilon}})\bm{\gamma}^{*}(\bm{x}_t^{\epsilon},\mathscr{L}_{\bm{x}_t^{\epsilon}}) = \bm{\sigma}(\bm{x}_t^{\epsilon})\bm{\sigma}^{*}(\bm{x}_t^{\epsilon}).
\end{eqnarray}
For the second term $I_{2}$, we have
\begin{eqnarray}\label{i2eq}
    I_{2} &=& \int_{0}^{t} \epsilon(\bm{v}_{s}^{\epsilon})_{j}\widetilde{\mathbb {E}}[(\widetilde{\bm{v}_{s}^{\epsilon}})_{l}(\partial_{\mu}\bm{\gamma}_{ij}^{-1}( \bm{x}_{s}^{\epsilon},\mathscr{L}_{\bm{x}_{s}^{\epsilon}})(\widetilde{\bm{x}_{s}^{\epsilon}}))_{l}] {d}s
\nonumber\\
    &=& \widetilde{\mathbb {E}}[\int_{0}^{t} \epsilon(\bm{v}_{s}^{\epsilon})_{j}(\widetilde{\bm{v}_{s}^{\epsilon}})_{l}(\partial_{\mu}\bm{\gamma}_{ij}^{-1}( \bm{x}_{s}^{\epsilon},\mathscr{L}_{\bm{x}_{s}^{\epsilon}})(\widetilde{\bm{x}_{s}^{\epsilon}}))_{l} {d}s],
    \nonumber
\end{eqnarray}
where the product $\epsilon(\bm{v}_{s}^{\epsilon})_{j}(\widetilde{\bm{v}_{s}^{\epsilon}})_{l}$ in the above integral is the $(j,l)$-entry of the (outer product) matrix $\epsilon \bm{v}_{s}^{\epsilon}(\widetilde{\bm{v}_{s}^{\epsilon}})^{*}$.
Using the method in \cite{shamgv:smklim}, we express this
matrix as a solution to a Sylevster equation.
Using the It\^{o} product formula, we obatin
\begin{eqnarray}\label{vveq}
     {d}[\epsilon \bm{v}_{t}^{\epsilon}(\epsilon \widetilde{\bm{v}_{t}^{\epsilon}})^{*}] &=&  {d}(\epsilon \bm{v}_{t}^{\epsilon})(\epsilon \widetilde{\bm{v}_{t}^{\epsilon}})^{*}+\epsilon \bm{v}_{t}^{\epsilon} {d}(\epsilon \widetilde{\bm{v}_{t}^{\epsilon}})^{*}+ {d}(\epsilon \bm{v}_{t}^{\epsilon}) {d}(\epsilon\widetilde{\bm{v}_{t}^{\epsilon}})^{*}
    \nonumber\\
    &=& \Big(F(\bm{x}_t^{\epsilon}) {d}t-\bm{\gamma}( \bm{x}_t^{\epsilon},\mathscr{L}_{\bm{x}_t^{\epsilon}}) \bm{v}_{t}^{\epsilon} {d}t+\bm{\sigma}(\bm{x}_t^{\epsilon}) {d}\bm{W}_{t}\Big)(\epsilon \widetilde{\bm{v}_{t}^{\epsilon}})^{*}
    \nonumber\\
    &+&\epsilon \bm{v}_{t}^{\epsilon}\Big(F^{*}(\widetilde{\bm{x}_t^{\epsilon}}) {d}t-(\widetilde{\bm{v}_{t}^{\epsilon}})^{*}\bm{\gamma}^{*}( \widetilde{\bm{x}_t^{\epsilon}},\mathscr{L}_{\bm{x}_t^{\epsilon}}) \mathrm{d}t+(\bm{\sigma}(\widetilde{\bm{x}_t^{\epsilon}})\mathrm{d}\widetilde{\bm{W}_{t}})^{*}\Big)
   \nonumber\\
    &+&  { \bm{\sigma}(\bm{x}_t^{\epsilon})\bm{\sigma}^{*}(\widetilde{\bm{x}_t^{\epsilon}}) {d}[W,\widetilde{W}](t)}.
\end{eqnarray}
Defining
\begin{eqnarray}\label{u1eq}
    \mathrm{d}\widehat{U_{t}^{\epsilon}} = \epsilon \bm{v}_{t}^{\epsilon}F^{*}(\widetilde{\bm{x}_t^{\epsilon}}) {d}t+\epsilon \bm{v}_{t}^{\epsilon}(\bm{\sigma}(\widetilde{\bm{x}_t^{\epsilon}}) {d}\widetilde{\bm{W}_{t}})^{*},
\end{eqnarray}
and
 \begin{eqnarray}\label{u2eq}
      {d}\widehat{\widehat{U_{t}^{\epsilon}}} = F(\bm{x}_t^{\epsilon})(\epsilon \widetilde{\bm{v}_{t}^{\epsilon}})^{*} {d}t+\bm{\sigma}(\bm{x}_t^{\epsilon}) {d}\bm{W}_{t}(\epsilon \widetilde{\bm{v}_{t}^{\epsilon}})^{*},
 \end{eqnarray}
we can rewrite equation (\ref{vveq}) as
\begin{eqnarray}\label{matrixeq1}
    &&(-\bm{\gamma}( \bm{x}_t^{\epsilon},\mathscr{L}_{\bm{x}_t^{\epsilon}})) \epsilon \bm{v}_{t}^{\epsilon}(\widetilde{\bm{v}_{t}^{\epsilon}})^{*} {d}t-\epsilon \bm{v}_{t}^{\epsilon}(\widetilde{\bm{v}_{t}^{\epsilon}})^{*}\bm{\gamma}^{*}( \widetilde{\bm{x}_t^{\epsilon}},\mathscr{L}_{\bm{x}_t^{\epsilon}})  {d}t
    \nonumber\\
    &=&  {d}[\epsilon \bm{v}_{t}^{\epsilon}(\epsilon \widetilde{\bm{v}_{t}^{\epsilon}})^{*}] - \bm{\sigma}(\bm{x}_t^{\epsilon})\bm{\sigma}^{*}(\widetilde{\bm{x}_t^{\epsilon}}) {d}t- {d}\widehat{U_{t}^{\epsilon}}- {d}\widehat{\widehat{U_{t}^{\epsilon}}}.
\end{eqnarray}
Denoting $\epsilon \bm{v}_{t}^{\epsilon}(\widetilde{\bm{v}_{t}^{\epsilon}})^{*} {d}t$ by $Y$, $-\bm{\gamma}( \bm{x}_t^{\epsilon},\mathscr{L}_{\bm{x}_t^{\epsilon}})$ by $A$, $\bm{\gamma}^{*}( \widetilde{\bm{x}_t^{\epsilon}},\mathscr{L}_{\bm{x}_t^{\epsilon}})$ by $B$ and the right hand side of equation (\ref{matrixeq1}) by $C$, we obtain
\begin{eqnarray}\label{matrixeq2}
    AY-YB = C,
\end{eqnarray}
which is a Sylvester equation \cite{DDS}. By assumption \ref{a4}, the spectrum of $A$ and spectrum of $B$ are contained in the open left half plane and the open right half plane, respectively, therefore, by \cite{DDS} the Sylvester equation (\ref{matrixeq2}) has a unique solution and the solution is
\begin{eqnarray}
    \epsilon \bm{v}_{t}^{\epsilon}(\widetilde{\bm{v}_{t}^{\epsilon}})^{*} {d}t &=& -\int_{0}^{\infty} e^{Ay}Ce^{-By} {d}y,
    \nonumber\\
    &=& -\int_{0}^{\infty} e^{-\bm{\gamma}( \bm{x}_t^{\epsilon},\mathscr{L}_{\bm{x}_t^{\epsilon}})y}\Big( {d}[\epsilon \bm{v}_{t}^{\epsilon}(\epsilon \widetilde{\bm{v}_{t}^{\epsilon}})^{*}] - \bm{\sigma}(\bm{x}_t^{\epsilon})\bm{\sigma}^{*}(\widetilde{\bm{x}_t^{\epsilon}}) {d}t- {d}\widehat{U_{t}^{\epsilon}}- {d}\widehat{\widehat{U_{t}^{\epsilon}}}\Big)e^{-\bm{\gamma}^{*}( \widetilde{\bm{x}_t^{\epsilon}},\mathscr{L}_{\bm{x}_t^{\epsilon}})y} {d}y
    \nonumber\\
    &=& -\int_{0}^{\infty} e^{-\bm{\gamma}( \bm{x}_t^{\epsilon},\mathscr{L}_{\bm{x}_t^{\epsilon}})y} {d}[\epsilon \bm{v}_{t}^{\epsilon}(\epsilon \widetilde{\bm{v}_{t}^{\epsilon}})^{*}]e^{-\bm{\gamma}^{*}( \widetilde{\bm{x}_t^{\epsilon}},\mathscr{L}_{\bm{x}_t^{\epsilon}})y} {d}y
    \nonumber\\
    &&+  \int_{0}^{\infty} e^{-\bm{\gamma}( \bm{x}_t^{\epsilon},\mathscr{L}_{\bm{x}_t^{\epsilon}})y}(\bm{\sigma}(\bm{x}_t^{\epsilon})\bm{\sigma}^{*}(\widetilde{\bm{x}_t^{\epsilon}}) {d}t)e^{-\bm{\gamma}^{*}( \widetilde{\bm{x}_t^{\epsilon}},\mathscr{L}_{\bm{x}_t^{\epsilon}})y} {d}y
    \nonumber\\
    &&+ \int_{0}^{\infty} e^{-\bm{\gamma}( \bm{x}_t^{\epsilon},\mathscr{L}_{\bm{x}_t^{\epsilon}})y}\Big( {d}\widehat{U_{t}^{\epsilon}}+ {d}\widehat{\widehat{U_{t}^{\epsilon}}}\Big)e^{-\bm{\gamma}^{*}( \widetilde{\bm{x}_t^{\epsilon}},\mathscr{L}_{\bm{x}_t^{\epsilon}})y} {d}y
    \nonumber\\
    &\triangleq&  {d}\widetilde{C_{t}^{1}}+ {d}\widetilde{C_{t}^{2}}+ {d}\widetilde{C_{t}^{3}}.
\end{eqnarray}
For the first term 
\begin{eqnarray}
     {d}(\widetilde{C_{t}^{1}})_{ij} = -\int_{0}^{\infty} (e^{-\bm{\gamma}( \bm{x}_t^{\epsilon},\mathscr{L}_{\bm{x}_t^{\epsilon}})y})_{ik_{1}}(e^{-\bm{\gamma}^{*}( \widetilde{\bm{x}_t^{\epsilon}},\mathscr{L}_{\bm{x}_t^{\epsilon}})y})_{k_{2}j} {d}y {d}[(\epsilon \bm{v}_{t}^{\epsilon})_{k_{1}}(\epsilon \widetilde{\bm{v}_{t}^{\epsilon}})^{*}_{k_{2}}],
\end{eqnarray}
where the integral exists and is finite for all $t \in [0, T]$.
For the second term $ {d}\widetilde{C_{t}^{2}}$, $ {d}\widetilde{C_{t}^{2}} = \widetilde{J}(\bm{x}_t^{\epsilon},\widetilde{\bm{x}_t^{\epsilon}},\mathscr{L}_{\bm{x}_t^{\epsilon}}) {d}t$, where $\widetilde{J}:\mathbb{R}^{d}\times\mathbb{R}^{d}\times\mathscr{P}_{2}\rightarrow\mathbb{R}^{d\times d}$ is the solution to the following Sylvester equation
\begin{eqnarray}
    -\bm{\gamma}( \bm{x}_t^{\epsilon},\mathscr{L}_{\bm{x}_t^{\epsilon}})\widetilde{J}(\bm{x}_t^{\epsilon},\widetilde{\bm{x}_t^{\epsilon}},\mathscr{L}_{\bm{x}_t^{\epsilon}})-\widetilde{J}(\bm{x}_t^{\epsilon},\widetilde{\bm{x}_t^{\epsilon}},\mathscr{L}_{\bm{x}_t^{\epsilon}})\bm{\gamma}^{*}( \widetilde{\bm{x}_t^{\epsilon}},\mathscr{L}_{\bm{x}_t^{\epsilon}}) = -\bm{\sigma}(\bm{x}_t^{\epsilon})\bm{\sigma}^{*}(\widetilde{\bm{x}_t^{\epsilon}}).
\end{eqnarray}
For the third term $ {d}\widetilde{C_{t}^{3}}$, using the equation (\ref{u1eq}) and (\ref{u2eq}) for $\widehat{U_{t}^{\epsilon}}$, $\widehat{\widehat{U_{t}^{\epsilon}}}$, the entries of $ {d}\widetilde{C_{t}^{3}}$ can be written as
\begin{eqnarray}
     {d}(\widetilde{C_{t}^{3}})_{ij} &=& -\int_{0}^{\infty} (e^{-\bm{\gamma}( \bm{x}_t^{\epsilon},\mathscr{L}_{\bm{x}_t^{\epsilon}})y})_{ik_{1}} (e^{-\bm{\gamma}^{*}( \widetilde{\bm{x}_t^{\epsilon}},\mathscr{L}_{\bm{x}_t^{\epsilon}})y})_{k_{2}j} {d}y
    \nonumber\\
    &&~~~~~~\times\Big((\epsilon \bm{v}_{t}^{\epsilon}F^{*}(\widetilde{\bm{x}_t^{\epsilon}}))_{k_{1}k_{2}} {d}t+(\epsilon \bm{v}_{t}^{\epsilon}(\bm{\sigma}(\widetilde{\bm{x}_t^{\epsilon}}) {d}\widetilde{\bm{W}_{t}})^{*})_{k_{1}k_{2}}
    \nonumber\\
    &&~~~~~~~~~~~~+ (F(\bm{x}_t^{\epsilon})(\epsilon \widetilde{\bm{v}_{t}^{\epsilon}})^{*})_{k_{1}k_{2}} {d}t+(\bm{\sigma}(\bm{x}_t^{\epsilon}) {d}\bm{W}_{t}(\epsilon \widetilde{\bm{v}_{t}^{\epsilon}})^{*})_{k_{1}k_{2}}\Big).
\end{eqnarray}
We substitute the expression for  $\epsilon \bm{v}_{t}^{\epsilon}(\widetilde{\bm{v}_{t}^{\epsilon}})^{*}$ back into equation (\ref{i2eq}) and obtain
\begin{eqnarray}\label{i2laeq}
    I_{2} &=& \widetilde{\mathbb {E}}\Big[\int_{0}^{t} (\partial_{\mu}\bm{\gamma}_{ij}^{-1}( \bm{x}_{s}^{\epsilon},\mathscr{L}_{\bm{x}_{s}^{\epsilon}})(\widetilde{\bm{x}_{s}^{\epsilon}}))_{l}\times
    \nonumber\\
    &&~~~~~~~~\Big(-\int_{0}^{\infty} (e^{-\bm{\gamma}( \bm{x}_t^{\epsilon},\mathscr{L}_{\bm{x}_t^{\epsilon}})y})_{jk_{1}}(e^{-\bm{\gamma}^{*}( \widetilde{\bm{x}_t^{\epsilon}},\mathscr{L}_{\bm{x}_t^{\epsilon}})y})_{k_{2}l} {d}y {d}[(\epsilon \bm{v}_{t}^{\epsilon})_{k_{1}}(\epsilon \widetilde{\bm{v}_{t}^{\epsilon}})^{*}_{k_{2}}]
    \nonumber\\
    &&~~~~~~~~~~+\widetilde{J}_{jl}(\bm{x}_t^{\epsilon},\widetilde{\bm{x}_t^{\epsilon}},\mathscr{L}_{\bm{x}_t^{\epsilon}}) {d}t
    \nonumber\\
    &&~~~~~~~~~~-\int_{0}^{\infty} (e^{-\bm{\gamma}( \bm{x}_t^{\epsilon},\mathscr{L}_{\bm{x}_t^{\epsilon}})y})_{jk_{1}} (e^{-\bm{\gamma}^{*}( \widetilde{\bm{x}_t^{\epsilon}},\mathscr{L}_{\bm{x}_t^{\epsilon}})y})_{k_{2}l} {d}y
    \nonumber\\
    &&~~~~~~~~~~~~~~~~\times[(\epsilon \bm{v}_{t}^{\epsilon}F^{*}(\widetilde{\bm{x}_t^{\epsilon}}))_{k_{1}k_{2}} {d}t+(\epsilon \bm{v}_{t}^{\epsilon}(\bm{\sigma}(\widetilde{\bm{x}_t^{\epsilon}}) {d}\widetilde{\bm{W}_{t}})^{*})_{k_{1}k_{2}}
    \nonumber\\
    &&~~~~~~~~~~~~~~~~~~~~~+ (F(\bm{x}_t^{\epsilon})(\epsilon \widetilde{\bm{v}_{t}^{\epsilon}})^{*})_{k_{1}k_{2}} {d}t+(\bm{\sigma}(\bm{x}_t^{\epsilon}) {d}\bm{W}_{t}(\epsilon \widetilde{\bm{v}_{t}^{\epsilon}})^{*})_{k_{1}k_{2}}]\Big)\Big].
\end{eqnarray}
Combining (\ref{xinteq}), (\ref{intpareq}), (\ref{itomeaeq}), (\ref{i1laeq}) and (\ref{i2laeq}), we have
\begin{eqnarray}
    \bm{x}_t^{\epsilon} = \bm{x}_t^{\epsilon,1}+\bm{x}_t^{\epsilon,2},
\end{eqnarray}
where $\bm{x}_t^{\epsilon,1}$ and $\bm{x}_t^{\epsilon,2}$ is in (\ref{xe1}) and (\ref{xe2}).

Secondly, we show that
$$    \mathbb{E}(\sup_{0\leq t \leq T}\mid \bm{x}_t^{\epsilon}-\bm{x}_{t}\mid^2) \leq C\sqrt{\epsilon}.$$
By equation~(\ref{maineq}) and (\ref{maineq-limit}),
\begin{eqnarray}
    (\bm{x}_t^{\epsilon})_{i}-(\bm{x}_{t})_{i} &=&  \int_{0}^{t}[\bm{\gamma}_{ij}^{-1}( \bm{x}_{s}^{\epsilon},\mathscr{L}_{\bm{x}_{s}^{\epsilon}})F_{j}(\bm{x}_{s}^{\epsilon})-\bm{\gamma}_{ij}^{-1}( \bm{x}_{s},\mathscr{L}_{\bm{x}_{s}})F_{j}(\bm{x}_{s})] {d}s
    \\
    &&+ \int_{0}^{t}[\bm{\gamma}_{ij}^{-1}( \bm{x}_{s}^{\epsilon},\mathscr{L}_{\bm{x}_{s}^{\epsilon}})\bm{\sigma}_{jk}(\bm{x}_{s}^{\epsilon})-\bm{\gamma}_{ij}^{-1}( \bm{x}_{s},\mathscr{L}_{\bm{x}_{s}})\bm{\sigma}_{jk}(\bm{x}_{s})] {d}(\bm{W}_{s})_{k}
    \nonumber\\
    &&+ \int_{0}^{t} \Big([\partial_{x_{l}}\bm{\gamma}_{ij}^{-1}( \bm{x}_{s}^{\epsilon},\mathscr{L}_{\bm{x}_{s}^{\epsilon}})]J_{jl}(\bm{x}_{s}^{\epsilon},\mathscr{L}_{\bm{x}_{s}^{\epsilon}})-[\partial_{x_{l}}\bm{\gamma}_{ij}^{-1}( \bm{x}_{s},\mathscr{L}_{\bm{x}_{s}})]J_{jl}(\bm{x}_{s},\mathscr{L}_{\bm{x}_{s}})\Big) {d}s
    \nonumber\\
    &&+\int_{0}^{t}\Big(\widetilde{\mathbb {E}}\Big[ (\partial_{\mu}\bm{\gamma}_{ij}^{-1}( \bm{x}_{s}^{\epsilon},\mathscr{L}_{\bm{x}_{s}^{\epsilon}})(\widetilde{\bm{x}_{s}^{\epsilon}}))_{l}\widetilde{J}_{jl}(\bm{x}_{s}^{\epsilon},
    \widetilde{\bm{x}_{s}^{\epsilon}},\mathscr{L}_{\bm{x}_{s}^{\epsilon}})\Big]
    \nonumber\\
        &&~~~~~~-\widetilde{\mathbb {E}}\Big[ (\partial_{\mu}\bm{\gamma}_{ij}^{-1}( \bm{x}_{s},\mathscr{L}_{\bm{x}_{s}})(\widetilde{\bm{x}_{s}}))_{l}\widetilde{J}_{jl}(\bm{x}_{s},\widetilde{\bm{x}_{s}},\mathscr{L}_{\bm{x}_{s}})\Big]\Big) {d}s
    \nonumber
 \end{eqnarray}
 \begin{eqnarray}
    &&+(-\bm{\gamma}_{ij}^{-1}( \bm{x}_t^{\epsilon},\mathscr{L}_{\bm{x}_t^{\epsilon}})\epsilon (\bm{v}_{t}^{\epsilon})_{j})+\bm{\gamma}_{ij}^{-1}( \bm{x}_{0},\mathscr{L}_{\bm{x}_{0}})\epsilon (\bm{v}_{0})_{j}
    \nonumber\\
    &&+(U_{t}^{\epsilon})_{i}+\widetilde{\mathbb{E}}\Big[\int_{0}^{t}(\partial_{\mu}\bm{\gamma}_{ij}^{-1}( \bm{x}_{s}^{\epsilon},\mathscr{L}_{\bm{x}_{s}^{\epsilon}})(\widetilde{\bm{x}_{s}^{\epsilon}}))_{l}\times
    \nonumber\\
    &&~~~~~~~~\Big(-\int_{0}^{\infty} (e^{-\bm{\gamma}( \bm{x}_t^{\epsilon},\mathscr{L}_{\bm{x}_t^{\epsilon}})y})_{jk_{1}} (e^{-\bm{\gamma}^{*}( \widetilde{\bm{x}_t^{\epsilon}},\mathscr{L}_{\bm{x}_t^{\epsilon}})y})_{k_{2}l} {d}y
    \nonumber\\
    &&~~~~~~~~~~~~~~~~\times[(\epsilon \bm{v}_{t}^{\epsilon}F^{*}(\widetilde{\bm{x}_t^{\epsilon}}))_{k_{1}k_{2}} {d}t+(\epsilon \bm{v}_{t}^{\epsilon}(\bm{\sigma}(\widetilde{\bm{x}_t^{\epsilon}}) {d}\widetilde{\bm{W}_{t}})^{*})_{k_{1}k_{2}}
    \nonumber\\
    &&~~~~~~~~~~~~~~~~~~~~~+ (F(\bm{x}_t^{\epsilon})(\epsilon \widetilde{\bm{v}_{t}^{\epsilon}})^{*})_{k_{1}k_{2}} {d}t+(\bm{\sigma}(\bm{x}_t^{\epsilon})\mathrm{d}\bm{W}_{t}(\epsilon \widetilde{\bm{v}_{t}^{\epsilon}})^{*})_{k_{1}k_{2}}]\Big)\Big]
    \nonumber\\
    &&+\int_{0}^{t} [\partial_{x_{l}}\bm{\gamma}_{ij}^{-1}( \bm{x}_{s}^{\epsilon},\mathscr{L}_{\bm{x}_{s}^{\epsilon}})]\times
    \nonumber\\
    &&~~~~~~~~\Big[-\int_{0}^{\infty}(e^{-\bm{\gamma}( \bm{x}_{s}^{\epsilon},\mathscr{L}_{\bm{x}_{s}^{\epsilon}})y})_{jk_{1}}(e^{-\bm{\gamma}^{*}( \bm{x}_{s}^{\epsilon},\mathscr{L}_{\bm{x}_{s}^{\epsilon}})y})_{k_{2}l} {d}y\Big] {d}[(\epsilon \bm{v}_{s}^{\epsilon})_{k_{1}}(\epsilon \bm{v}_{s}^{\epsilon})_{k_{2}}]
    \nonumber\\
    &&~~~~+\widetilde{\mathbb {E}}\Big[\int_{0}^{t}(\partial_{\mu}\bm{\gamma}_{ij}^{-1}( \bm{x}_{s}^{\epsilon},\mathscr{L}_{\bm{x}_{s}^{\epsilon}})(\widetilde{\bm{x}_{s}^{\epsilon}}))_{l}\times
    \nonumber\\
    &&~~~~~~~~\Big(-\int_{0}^{\infty} (e^{-\bm{\gamma}( \bm{x}_t^{\epsilon},\mathscr{L}_{\bm{x}_t^{\epsilon}})y})_{jk_{1}}(e^{-\bm{\gamma}^{*}( \widetilde{\bm{x}_t^{\epsilon}},\mathscr{L}_{\bm{x}_t^{\epsilon}})y})_{k_{2}l} {d}y {d}[(\epsilon \bm{v}_{t}^{\epsilon})_{k_{1}}(\epsilon \widetilde{\bm{v}_{t}^{\epsilon}})^{*}_{k_{2}}]\Big)\Big]
    \nonumber\\
    &:=& \sum_{i=1}^{7} y_{i}.
    \nonumber
\end{eqnarray}
By Cauchy-schwarz's inequality, then
\begin{eqnarray}
   \mathbb {E}\mid (\bm{x}_t^{\epsilon})_{i}-(\bm{x}_{t})_{i}\mid^{2}\leq C\sum_{i=1}^{7}\mathbb {E}\mid y_{i}\mid^{2} := \sum_{i=1}^{7} I_{i}(t).
\end{eqnarray}
For $I_{1}$, since by assumption \ref{a4},  $\parallel \bm{\gamma}(\bm{x},\mu)\parallel\geq C_{\lambda_{\bm{\gamma}}}$, then
\begin{eqnarray}\label{11}
    \parallel \bm{\gamma}^{-1}(\bm{x},\mu)\parallel\leq \frac{1}{C_{\lambda_{\bm{\gamma}}}},
\end{eqnarray}
by assumption \ref{a3}, we have
\begin{eqnarray}\label{12}
\parallel \bm{\gamma}(\bm{x}_{1},\mu_{1})-\bm{\gamma}(\bm{x}_{2},\mu_{2})\parallel\leq C_{T}\Big[\mid \bm{x}_{1}-\bm{x}_{2}\mid+\mathbb{W}(\mu_{1},\mu_{2})\Big],
\end{eqnarray}
hence, combining (\ref{11}) and (\ref{12}) with H\"{o}lder inequality, Cauchy-schwarz's inequality and assumptions \ref{a1}-\ref{a4}, we have for any $t\in[0,T]$
\begin{eqnarray}
    I_{1}&=& \mathbb {E}\Big| \int_{0}^{t}[\bm{\gamma}_{ij}^{-1}( \bm{x}_{s}^{\epsilon},\mathscr{L}_{\bm{x}_{s}^{\epsilon}})F_{j}(\bm{x}_{s}^{\epsilon})-\bm{\gamma}_{ij}^{-1}( \bm{x}_{s},\mathscr{L}_{\bm{x}_{s}})F_{j}(\bm{x}_{s})] {d}s\Big|^{2}
    \nonumber\\
    &\leq& T\mathbb {E}\int_{0}^{t}\Big|[\bm{\gamma}_{ij}^{-1}( \bm{x}_{s}^{\epsilon},\mathscr{L}_{\bm{x}_{s}^{\epsilon}})F_{j}(\bm{x}_{s}^{\epsilon})-\bm{\gamma}_{ij}^{-1}( \bm{x}_{s},\mathscr{L}_{\bm{x}_{s}})F_{j}(\bm{x}_{s})]\Big|^{2} {d}s
    \nonumber\\
         &\leq& T\mathbb {E}\int_{0}^{t}\Big|[\bm{\gamma}_{ij}^{-1}( \bm{x}_{s}^{\epsilon},\mathscr{L}_{\bm{x}_{s}^{\epsilon}})-\bm{\gamma}_{ij}^{-1}( \bm{x}_{s},\mathscr{L}_{\bm{x}_{s}})]F_{j}(\bm{x}_{s}^{\epsilon})\Big|^{2} {d}s
    \nonumber\\
   &&+  T\mathbb {E}\int_{0}^{t}\Big|\bm{\gamma}_{ij}^{-1}( \bm{x}_{s},\mathscr{L}_{\bm{x}_{s}})[F_{j}(\bm{x}_{s}^{\epsilon})-F_{j}(\bm{x}_{s})]\Big|^{2} {d}s
    \nonumber\\
   &\leq&  \frac{C_{T}}{C_{\lambda_{\bm{\gamma}}}}\mathbb {E}\int_{0}^{t}\Big(|\bm{x}_{s}^{\epsilon}-\bm{x}_{s}|^{2}+W^{2}(\mathscr{L}_{\bm{x}_{s}}^{\epsilon},\mathscr{L}_{\bm{x}_{s}})\Big) {d}s+\frac{T}{C_{\lambda_{\bm{\gamma}}}}\mathbb {E}\int_{0}^{t}|\bm{x}_{s}^{\epsilon}-\bm{x}_{s}|^{2} {d}s
    \nonumber\\
    &\leq& \frac{C_{T}}{C_{\lambda_{\bm{\gamma}}}} \int_{0}^{t}\mathbb {E}|\bm{x}_{s}^{\epsilon}-\bm{x}_{s}|^{2} {d}s.
\end{eqnarray}
For $I_{2}$, by Burkh\"{o}lder inequality and assumptions \ref{a1} and \ref{a2}, we have for any $t\in[0,T]$
\begin{eqnarray}
    I_{2}&=& \mathbb {E}\Big| \int_{0}^{t}[\bm{\gamma}_{ij}^{-1}( \bm{x}_{s}^{\epsilon},\mathscr{L}_{\bm{x}_{s}^{\epsilon}})\bm{\sigma}_{jk}(\bm{x}_{s}^{\epsilon})-\bm{\gamma}_{ij}^{-1}( \bm{x}_{s},\mathscr{L}_{\bm{x}_{s}})\bm{\sigma}_{jk}(\bm{x}_{s})] {d}(\bm{W}_{s})_{k}\Big|^{2}
    \nonumber\\
    &\leq& T\mathbb {E}\int_{0}^{t}\Big|[\bm{\gamma}_{ij}^{-1}( \bm{x}_{s}^{\epsilon},\mathscr{L}_{\bm{x}_{s}^{\epsilon}})\bm{\sigma}_{jk}(\bm{x}_{s}^{\epsilon})-\bm{\gamma}_{ij}^{-1}( \bm{x}_{s},\mathscr{L}_{\bm{x}_{s}})\bm{\sigma}_{jk}(\bm{x}_{s})]\Big|^{2} {d}s
    \nonumber\\
     &\leq& T\mathbb {E}\int_{0}^{t}\Big|[\bm{\gamma}_{ij}^{-1}( \bm{x}_{s}^{\epsilon},\mathscr{L}_{\bm{x}_{s}^{\epsilon}})-\bm{\gamma}_{ij}^{-1}( \bm{x}_{s},\mathscr{L}_{\bm{x}_{s}})]\bm{\sigma}_{jk}(\bm{x}_{s}^{\epsilon})\Big|^{2} {d}s
    \nonumber\\
     &&+ T\mathbb {E}\int_{0}^{t}\Big|\bm{\gamma}_{ij}^{-1}( \bm{x}_{s},\mathscr{L}_{\bm{x}_{s}})[\bm{\sigma}_{jk}(\bm{x}_{s}^{\epsilon})-\bm{\sigma}_{jk}(\bm{x}_{s})]\Big|^{2} {d}s
    \nonumber\\
    &\leq& \frac{C_{T}}{C_{\lambda_{\bm{\gamma}}}} \mathbb {E}\int_{0}^{t}\Big(|\bm{x}_{s}^{\epsilon}-\bm{x}_{s}|^{2}+W^{2}(\mathscr{L}_{\bm{x}_{s}}^{\epsilon},\mathscr{L}_{\bm{x}_{s}})\Big) {d}s
    \nonumber\\
    &\leq& \frac{C_{T}}{C_{\lambda_{\bm{\gamma}}}} \int_{0}^{t}\mathbb {E}|\bm{x}_{s}^{\epsilon}-\bm{x}_{s}|^{2} {d}s.
\end{eqnarray}
For $I_{3}$, note that
$$\partial_{x}\bm{\gamma}_{ij}^{-1}( \bm{x}_{s}^{\epsilon},\mathscr{L}_{\bm{x}_{s}^{\epsilon}}) = - \bm{\gamma}_{ij}^{-1}( \bm{x}_{s}^{\epsilon},\mathscr{L}_{\bm{x}_{s}^{\epsilon}})\Big(\partial_{x}\bm{\gamma}_{ij}( \bm{x}_{s}^{\epsilon},\mathscr{L}_{\bm{x}_{s}^{\epsilon}})\Big)\bm{\gamma}_{ij}^{-1}( \bm{x}_{s}^{\epsilon},\mathscr{L}_{\bm{x}_{s}^{\epsilon}}),$$
and by assumption \ref{a3} and \ref{a4}, we have 
\begin{eqnarray}\label{gaminvcon}
   \mid[\partial_{x_{l}}\bm{\gamma}_{ij}^{-1}( \bm{x}_{s}^{\epsilon},\mathscr{L}_{\bm{x}_{s}^{\epsilon}})]-[\partial_{x_{l}}\bm{\gamma}_{ij}^{-1}( \bm{x}_{s},\mathscr{L}_{\bm{x}_{s}})]\mid \leq C_{T}(\mid \bm{x}_{s}^{\epsilon}-\bm{x}_{s}\mid+W(\mathscr{L}_{\bm{x}_{s}^{\epsilon}},\mathscr{L}_{\bm{x}_{s}})).
\end{eqnarray}
For $J(\bm{x}_t^{\epsilon},\mathscr{L}_{\bm{x}_t^{\epsilon}})$, by assumptions \ref{a2}, \ref{a4}  and Lemma 4.4.2 in \cite{HP:ordieq}, we have
\begin{eqnarray}
    \parallel J(\bm{x}_t^{\epsilon},\mathscr{L}_{\bm{x}_t^{\epsilon}})\parallel^{2} &=& \parallel\int_{0}^{\infty} e^{-\bm{\gamma}( \bm{x}_t^{\epsilon},\mathscr{L}_{\bm{x}_t^{\epsilon}})y}(\bm{\sigma}(\bm{x}_t^{\epsilon})\bm{\sigma}^{*}(\bm{x}_t^{\epsilon}))e^{-\bm{\gamma}^{*}( \bm{x}_t^{\epsilon},\mathscr{L}_{\bm{x}_t^{\epsilon}})y} {d}y\parallel^{2}
    \nonumber\\
    &\leq& \int_{0}^{\infty} \parallel e^{-\bm{\gamma}( \bm{x}_t^{\epsilon},\mathscr{L}_{\bm{x}_t^{\epsilon}})y}\parallel^{2}\parallel(\bm{\sigma}(\bm{x}_t^{\epsilon})\bm{\sigma}^{*}(\bm{x}_t^{\epsilon}))\parallel^{2}\parallel e^{-\bm{\gamma}^{*}( \bm{x}_t^{\epsilon},\mathscr{L}_{\bm{x}_t^{\epsilon}})y}\parallel^{2} {d}y
    \nonumber\\
    &\leq& 2C_{T}\int_{0}^{\infty} e^{-C_{\lambda_{\bm{\gamma}}}y}\mathrm{d}y,
    \nonumber\\
    &\leq& \frac{C_{T}}{C_{\lambda_{\bm{\gamma}}}},
\end{eqnarray}
similarly,
\begin{eqnarray}\label{J-1}
    \parallel \widetilde{J}(\bm{x}_t^{\epsilon},\widetilde{x}_{t}^{\epsilon},\mathscr{L}_{\bm{x}_t^{\epsilon}})\parallel^{2}\leq \frac{C_{T}}{C_{\lambda_{\bm{\gamma}}}}.
\end{eqnarray}

Therefore, by H\"{o}lder inequality, Cauchy-schwarz's inequality, boundness of $\partial_{x}\bm{\gamma}^{-1}( \bm{x}_{s}^{\epsilon},\mathscr{L}_{\bm{x}_{s}^{\epsilon}})$, $J(\bm{x}_t^{\epsilon},\mathscr{L}_{\bm{x}_t^{\epsilon}})$, and (\ref{gaminvcon}), we have
\begin{eqnarray}
    I_{3} &=& \mathbb {E}\Big|\int_{0}^{t} \Big([\partial_{x_{l}}\bm{\gamma}_{ij}^{-1}( \bm{x}_{s}^{\epsilon},\mathscr{L}_{\bm{x}_{s}^{\epsilon}})]J_{jl}(\bm{x}_{s}^{\epsilon},\mathscr{L}_{\bm{x}_{s}^{\epsilon}})-[\partial_{x_{l}}\bm{\gamma}_{ij}^{-1}( \bm{x}_{s},\mathscr{L}_{\bm{x}_{s}})]J_{jl}(\bm{x}_{s},\mathscr{L}_{\bm{x}_{s}})\Big) {d}s\Big|^{2}
    \nonumber\\
    &\leq& T\mathbb {E}\int_{0}^{t} \Big|\Big([\partial_{x_{l}}\bm{\gamma}_{ij}^{-1}( \bm{x}_{s}^{\epsilon},\mathscr{L}_{\bm{x}_{s}^{\epsilon}})]-[\partial_{x_{l}}\bm{\gamma}_{ij}^{-1}( \bm{x}_{s},\mathscr{L}_{\bm{x}_{s}})]\Big)J_{jl}(\bm{x}_{s}^{\epsilon},\mathscr{L}_{\bm{x}_{s}^{\epsilon}})\Big|^{2} {d}s
    \nonumber\\
    &&+  {\mathbb {E}\int_{0}^{t} \Big|[\partial_{x_{l}}\bm{\gamma}_{ij}^{-1}( \bm{x}_{s},\mathscr{L}_{\bm{x}_{s}})]\Big(J_{jl}(\bm{x}_{s}^{\epsilon},\mathscr{L}_{\bm{x}_{s}^{\epsilon}})-J_{jl}(\bm{x}_{s},\mathscr{L}_{\bm{x}_{s}})\Big)\Big|^{2} {d}s}
    \nonumber\\
    &\leq& \frac{C_{T}}{C_{\lambda_{\bm{\gamma}}}} \mathbb {E}\int_{0}^{t}\Big(\mid \bm{x}_{s}^{\epsilon}-\bm{x}_{s}\mid^{2}+W^{2}(\mathscr{L}_{\bm{x}_{s}^{\epsilon}},\mathscr{L}_{\bm{x}_{s}})\Big) {d}s
    \nonumber\\
    &&+
     {C_T}{\mathbb {E}\int_{0}^{t}\Big|
     J_{jl}(\bm{x}_{s}^{\epsilon},\mathscr{L}_{\bm{x}_{s}^{\epsilon}})-J_{jl}(\bm{x}_{s},\mathscr{L}_{\bm{x}_{s}})\Big|^{2} {d}s}.
\end{eqnarray}
Using equation~(\ref{thm-new-1}),  then
\begin{eqnarray}
&&\Big| J(\bm{x}_t^{\epsilon},\mathscr{L}_{\bm{x}_t^{\epsilon}})- J(\bm{x}_t ,\mathscr{L}_{\bm{x}_t})\Big|^2  \nonumber\\
&=& \Big| \int_{0}^{\infty} e^{-\bm{\gamma}( \bm{x}_t^{\epsilon},\mathscr{L}_{\bm{x}_t^{\epsilon}})y}(\bm{\sigma}(\bm{x}_t^{\epsilon})\bm{\sigma}^{*}(\bm{x}_t^{\epsilon}))e^{-\bm{\gamma}^{*}( \bm{x}_t^{\epsilon},\mathscr{L}_{\bm{x}_t^{\epsilon}})y}-  e^{-\bm{\gamma}( \bm{x}_t ,\mathscr{L}_{\bm{x}_t })y}(\bm{\sigma}(\bm{x}_t )\bm{\sigma}^{*}(\bm{x}_t ))e^{-\bm{\gamma}^{*}( \bm{x}_t ,\mathscr{L}_{\bm{x}_t })y} {d}y \Big|^2 \nonumber\\
&\leq& 2\Big| \int_{0}^{\infty} e^{-\bm{\gamma}( \bm{x}_t^{\epsilon},\mathscr{L}_{\bm{x}_t^{\epsilon}})y}(\bm{\sigma}(\bm{x}_t^{\epsilon})\bm{\sigma}^{*}(\bm{x}_t^{\epsilon})-\bm{\sigma}( \bm{x}_t )\bm{\sigma}^{*}(\bm{x}_t )   )e^{-\bm{\gamma}^{*}( \bm{x}_t^{\epsilon},\mathscr{L}_{\bm{x}_t^{\epsilon}})y} dy \Big|^2 \nonumber\\
&&+  2 \Big| \int_{0}^{\infty} e^{-\bm{\gamma}( \bm{x}_t^{\epsilon},\mathscr{L}_{\bm{x}_t^{\epsilon}})y}(\bm{\sigma}(\bm{x}_t )\bm{\sigma}^{*}(\bm{x}_t ))e^{-\bm{\gamma}^{*}( \bm{x}_t^{\epsilon},\mathscr{L}_{\bm{x}_t^{\epsilon}})y}-  e^{-\bm{\gamma}( \bm{x}_t ,\mathscr{L}_{\bm{x}_t })y}(\bm{\sigma}(\bm{x}_t )\bm{\sigma}^{*}(\bm{x}_t ))e^{-\bm{\gamma}^{*}( \bm{x}_t ,\mathscr{L}_{\bm{x}_t })y} {d}y \Big|^2\nonumber\\
&\leq& \frac{2C_{T}}{C_{\lambda_{\bm{\gamma}}}} \Big|\bm{\sigma}(\bm{x}_t^{\epsilon})\bm{\sigma}^{*}(\bm{x}_t^{\epsilon})-\bm{\sigma}( \bm{x}_t )\bm{\sigma}^{*}(\bm{x}_t ) \Big|^2  \nonumber\\
&&+  2 C_T \Big| \int_{0}^{\infty} e^{-\bm{\gamma}( \bm{x}_t^{\epsilon},\mathscr{L}_{\bm{x}_t^{\epsilon}})y}e^{-\bm{\gamma}^{*}( \bm{x}_t^{\epsilon},\mathscr{L}_{\bm{x}_t^{\epsilon}})y}-  e^{-\bm{\gamma}( \bm{x}_t ,\mathscr{L}_{\bm{x}_t })y} e^{-\bm{\gamma}^{*}( \bm{x}_t ,\mathscr{L}_{\bm{x}_t })y} {d}y \Big|^2,
\end{eqnarray}
by assumption \ref{a2} and Lemma \ref{lem1}, we have
\begin{eqnarray}
&&\mathbb {E}\int_{0}^{t} \Big\|\bm{\sigma}(\bm{x}_s^{\epsilon})\bm{\sigma}^{*}(\bm{x}_s^{\epsilon})-\bm{\sigma}( \bm{x}_s )\bm{\sigma}^{*}(\bm{x}_s ) \Big\|^2 d s \nonumber\\
&\leq& 2 \mathbb {E}\int_{0}^{t} \Big\|(\bm{\sigma}(\bm{x}_s^{\epsilon})-\bm{\sigma}( \bm{x}_s ))\bm{\sigma}^{*}(\bm{x}_s^{\epsilon}) \Big\|^2 d s
+ 2 \mathbb {E}\int_{0}^{t} \Big\|\bm{\sigma}(\bm{x}_s^{\epsilon})(\bm{\sigma}^{*}(\bm{x}_s^{\epsilon})- \bm{\sigma}^{*}(\bm{x}_s )) \Big\|^2  d s\nonumber\\
&\leq&  C_T\mathbb {E}\int_{0}^{t} \Big|\bm{x}_s^{\epsilon}-\bm{x}_s  \Big|^2d s,
\end{eqnarray}
and
\begin{eqnarray}
&&\Big| \int_{0}^{\infty} (e^{-\bm{\gamma}( \bm{x}_t^{\epsilon},\mathscr{L}_{\bm{x}_t^{\epsilon}})y})_{ik_1} (e^{-\bm{\gamma}^{*}( \bm{x}_t^{\epsilon},\mathscr{L}_{\bm{x}_t^{\epsilon}})y})_{k_2j}- ( e^{-\bm{\gamma}( \bm{x}_t ,\mathscr{L}_{\bm{x}_t })y})_{ik_1}( e^{-\bm{\gamma}^{*}( \bm{x}_t ,\mathscr{L}_{\bm{x}_t })y} )_{k_2j}  {d}y \Big|^2   \nonumber\\
&\leq &  2 \Big| \int_{0}^{\infty} \left((e^{-\bm{\gamma}( \bm{x}_t^{\epsilon},\mathscr{L}_{\bm{x}_t^{\epsilon}})y})_{ik_1}-(e^{-\bm{\gamma}( \bm{x}_t,\mathscr{L}_{\bm{x}_t })y})_{ik_1} \right) (e^{-\bm{\gamma}^{*}( \bm{x}_t^{\epsilon},\mathscr{L}_{\bm{x}_t^{\epsilon}})y})_{k_2j} {d}y \Big|^2   \nonumber\\
&&+  2 \Big| \int_{0}^{\infty}( e^{-\bm{\gamma}( \bm{x}_t ,\mathscr{L}_{\bm{x}_t})y})_{ik_1}  \left( (e^{-\bm{\gamma}^{*}( \bm{x}_t^{\epsilon},\mathscr{L}_{\bm{x}_t^{\epsilon}})y} )_{k_2j} - (e^{-\bm{\gamma}^{*}( \bm{x}_t ,\mathscr{L}_{\bm{x}_t })y})_{k_2j}\right) {d}y \Big|^2   \nonumber\\
&\leq &  2 \Big| \int_{0}^{\infty} \left((e^{-\bm{\gamma}( \bm{x}_t^{\epsilon},\mathscr{L}_{\bm{x}_t^{\epsilon}})y})_{ik_1}-(e^{-\bm{\gamma}( \bm{x}_t,\mathscr{L}_{\bm{x}_t })y})_{ik_1} \right) {d}y \Big|^2  \nonumber\\
&&+  2 \Big| \int_{0}^{\infty}  \left( (e^{-\bm{\gamma}^{*}( \bm{x}_t^{\epsilon},\mathscr{L}_{\bm{x}_t^{\epsilon}})y} )_{k_2j} - (e^{-\bm{\gamma}^{*}( \bm{x}_t ,\mathscr{L}_{\bm{x}_t })y})_{k_2j}\right)  {d}y \Big|^2   \nonumber\\
&\leq & 2C^{-4}_{\lambda_\gamma}|\bm{\gamma}_{ik_1}( \bm{x}_t,\mathscr{L}_{\bm{x}_t })- \bm{\gamma}_{ik_1} ( \bm{x}_t^{\epsilon},\mathscr{L}_{\bm{x}_t^{\epsilon}})|^2
+2 C^{-4}_{\lambda_\gamma}|\bm{\gamma}_{k_2j}^{*}( \bm{x}_t,\mathscr{L}_{\bm{x}_t })_{ik_1}- \bm{\gamma}^{*}_{k_2j} ( \bm{x}_t^{\epsilon},\mathscr{L}_{\bm{x}_t^{\epsilon}})|^2 \nonumber\\
&\leq & 2C_TC^{-4}_{\lambda_\gamma} \left[|\bm{x}_t^{\epsilon}-\bm{x}_t|^2  + W^{2}(\mathscr{L}_{\bm{x}_t^{\epsilon}},\mathscr{L}_{\bm{x}_t})\right] \nonumber\\
&\leq & 4 C_TC^{-4}_{\lambda_\gamma}|\bm{x}_t^{\epsilon}-\bm{x}_t|^2 .
\end{eqnarray}
Hence,
\begin{eqnarray}
I_{3} \leq C_T \mathbb {E}\int_{0}^{t}\Big(\mid \bm{x}_{s}^{\epsilon}-\bm{x}_{s}\mid^{2}+W^{2}(\mathscr{L}_{\bm{x}_{s}^{\epsilon}},\mathscr{L}_{\bm{x}_{s}})\Big) {d}s
 \leq C_T \mathbb {E}\int_{0}^{t} \mid \bm{x}_{s}^{\epsilon}-\bm{x}_{s}\mid^{2}d s .
 \end{eqnarray}
By H\"{o}lder inequality and equation(\ref{J-1}), then using the same estimation method as $I_{3}$,
\begin{eqnarray}
    I_{4} &=& \mathbb {E}\Big|\widetilde{\mathbb {E}}\int_{0}^{t} \Big(\Big[ (\partial_{\mu}\bm{\gamma}_{ij}^{-1}( \bm{x}_{s}^{\epsilon},\mathscr{L}_{\bm{x}_{s}^{\epsilon}})(\widetilde{\bm{x}_{s}^{\epsilon}}))_{l}\widetilde{J}_{jl}(\bm{x}_{s}^{\epsilon},\widetilde{\bm{x}_{s}^{\epsilon}},\mathscr{L}_{\bm{x}_{s}^{\epsilon}})\Big]
    \nonumber\\
    &&~~~~~~~~~~~~~~~
    -\Big[ (\partial_{\mu}\bm{\gamma}_{ij}^{-1}( \bm{x}_{s},\mathscr{L}_{\bm{x}_{s}})(\widetilde{\bm{x}_{s}}))_{l}\widetilde{J}_{jl}(\bm{x}_{s},\widetilde{\bm{x}_{s}},\mathscr{L}_{\bm{x}_{s}})\Big]\Big)
     {d}s\Big|^{2}
    \nonumber\\
    &\leq& T \mathbb {E}(\widetilde{\mathbb {E}}\int_{0}^{t} \Big|\Big[ (\partial_{\mu}\bm{\gamma}_{ij}^{-1}( \bm{x}_{s}^{\epsilon},\mathscr{L}_{\bm{x}_{s}^{\epsilon}})(\widetilde{\bm{x}_{s}^{\epsilon}}))_{l}\widetilde{J}_{jl}
    (\bm{x}_{s}^{\epsilon},\widetilde{\bm{x}_{s}^{\epsilon}},\mathscr{L}_{\bm{x}_{s}^{\epsilon}})\Big]
    \nonumber\\
    &&~~~~~~~~~~~~~~~
    -\Big[ (\partial_{\mu}\bm{\gamma}_{ij}^{-1}( \bm{x}_{s},\mathscr{L}_{\bm{x}_{s}})(\widetilde{\bm{x}_{s}}))_{l}\widetilde{J}_{jl}(\bm{x}_{s},\widetilde{\bm{x}_{s}},\mathscr{L}_{\bm{x}_{s}})\Big]\Big|^{2} {d}s)
    \nonumber\\
    &\leq&  T \mathbb {E}(\widetilde{\mathbb {E}}\int_{0}^{t} \Big|\Big[ (\partial_{\mu}\bm{\gamma}_{ij}^{-1}( \bm{x}_{s}^{\epsilon},\mathscr{L}_{\bm{x}_{s}^{\epsilon}})(\widetilde{\bm{x}_{s}^{\epsilon}}))_{l}(\widetilde{J}_{jl}
    (\bm{x}_{s}^{\epsilon},\widetilde{\bm{x}_{s}^{\epsilon}},\mathscr{L}_{\bm{x}_{s}^{\epsilon}})- \widetilde{J}_{jl}
    (\bm{x}_{s},\widetilde{\bm{x}_{s} },\mathscr{L}_{\bm{x}_{s} })  )\Big]  \nonumber\\
    &&~~~~~~~~~~~~~~~
    +\Big[ ( (\partial_{\mu}\bm{\gamma}_{ij}^{-1}( \bm{x}_{s}^{\epsilon},\mathscr{L}_{\bm{x}_{s}^{\epsilon}})(\widetilde{\bm{x}_{s}^{\epsilon}}))_{l} - (\partial_{\mu}\bm{\gamma}_{ij}^{-1}( \bm{x}_{s},\mathscr{L}_{\bm{x}_{s}})(\widetilde{\bm{x}_{s}}))_{l} )\widetilde{J}_{jl}(\bm{x}_{s},\widetilde{\bm{x}_{s}},\mathscr{L}_{\bm{x}_{s}})\Big]\Big|^{2} {d}s)
    \nonumber
    \end{eqnarray}
    \begin{eqnarray}
    &\leq& C^2T \mathbb {E}(\widetilde{\mathbb {E}}\int_{0}^{t} \Big|(\widetilde{J}_{jl}
    (\bm{x}_{s}^{\epsilon},\widetilde{\bm{x}_{s}^{\epsilon}},\mathscr{L}_{\bm{x}_{s}^{\epsilon}})- \widetilde{J}_{jl}
    (\bm{x}_{s},\widetilde{\bm{x}_{s} },\mathscr{L}_{\bm{x}_{s} }) )\Big|^2 d s ) \nonumber\\
    &&~~~~~~~~~~~~~~~
    + \frac{T C_{T}}{C_{\lambda_{\bm{\gamma}}}}\mathbb {E}\int_{0}^{t}\Big(\mid \bm{x}_{s}^{\epsilon}-\bm{x}_{s}\mid^{2}+W^{2}(\mathscr{L}_{\bm{x}_{s}^{\epsilon}},\mathscr{L}_{\bm{x}_{s}})\Big) {d}s \nonumber\\
    &\leq&  C_T\mathbb {E}\int_{0}^{t} \Big|\bm{x}_s^{\epsilon}-\bm{x}_s  \Big|^2d s.
\end{eqnarray}
Since $\mid\bm{\gamma}_{ij}^{-1}(x,\mu)\mid\leq \frac{1}{C_{\lambda_{\bm{\gamma}}}}$, by H\"{o}lder inequality and Lemma \ref{pconver},
\begin{eqnarray}
    I_{5} &=& \mathbb {E}\Big(\sup_{0\leq t\leq T}\mid-\bm{\gamma}_{ij}^{-1}( \bm{x}_t^{\epsilon},\mathscr{L}_{\bm{x}_t^{\epsilon}})\epsilon (\bm{v}_{t}^{\epsilon})_{j}+\bm{\gamma}_{ij}^{-1}( \bm{x}_{0},\mathscr{L}_{\bm{x}_{0}})\epsilon (\bm{v}_{0})_{j}\mid\Big)^{2}
    \nonumber\\
    &\leq & 2\mathbb {E}\Big(\sup_{0\leq t\leq T}\mid-\bm{\gamma}_{ij}^{-1}( \bm{x}_t^{\epsilon},\mathscr{L}_{\bm{x}_t^{\epsilon}})\epsilon (\bm{v}_{t}^{\epsilon})_{j}\mid\Big)^{2}+2\mathbb {E}\Big(\sup_{0\leq t\leq T}\mid\bm{\gamma}_{ij}^{-1}( \bm{x}_{0},\mathscr{L}_{\bm{x}_{0}})\epsilon (\bm{v}_{0})_{j}\mid\Big)^{2}
    \nonumber\\
    &\leq &C_{T} \mathbb {E}\Big(\sup_{0\leq t\leq T}\mid\epsilon (\bm{v}_{t}^{\epsilon})_{j}\mid^{2}\Big)+ \mathbb {E}\Big(\sup_{0\leq t\leq T}\mid\epsilon (\bm{v}_{0})_{j}\mid^{2}\Big)
    \nonumber\\
        &\leq &C_{T}\Big[\mathbb {E}\Big(\sup_{0\leq t\leq T}\mid\epsilon (\bm{v}_{t}^{\epsilon})_{j}\mid^{2}\Big)^{2}\Big]^{1/2}+\Big[\mathbb {E}\Big(\sup_{0\leq t\leq T}\mid\epsilon (\bm{v}_{0})_{j}\mid^{2}\Big)^{2}\Big]^{1/2}
    \nonumber\\
&\leq &C_{T}\Big[\mathbb {E}\Big(\sup_{0\leq t\leq T}\mid\epsilon (\bm{v}_{t}^{\epsilon})\mid\Big)^{4}\Big]^{1/2}+\Big[\mathbb {E}\Big(\sup_{0\leq t\leq T}\mid\epsilon (\bm{v}_{0})\mid\Big)^{4}\Big]^{1/2}
    \nonumber\\
    &\leq& C_{T}\sqrt{\epsilon}.
\end{eqnarray}
By Doob's maximal inequality, Burkh\"{o}lder inequality,
\begin{eqnarray}
    \mathbb {E}\sup_{0\leq t\leq T}\Big|\int_{0}^{t}(\bm{\sigma}(\bm{x}_{s}^{\epsilon}) {d}\bm{W}_{s}(\epsilon \bm{v}_{s}^{\epsilon})^{*})_{k_{1}k_{2}}\Big|^{2}&\leq&  4\mathbb {E}\Big|\int_{0}^{T}(\bm{\sigma}(\bm{x}_{s}^{\epsilon}) {d}\bm{W}_{s}(\epsilon \bm{v}_{s}^{\epsilon})^{*})_{k_{1}k_{2}}\Big|^{2}
    \nonumber\\
    &\leq&4C \mathbb {E}\int_{0}^{T}\parallel\bm{\sigma}(\bm{x}_{s}^{\epsilon})\parallel^{2}|\epsilon \bm{v}_{s}^{\epsilon}|^{2} {d}s,
    \nonumber
\end{eqnarray}
and by assumption \ref{a2} and Lemma \ref{lem1}, we have
\begin{eqnarray}
 \parallel\bm{\sigma}(\bm{x}_{s}^{\epsilon})\parallel^{2}\leq C_{T},
\end{eqnarray}
and
\begin{eqnarray}
    &&\mathbb {E}|\epsilon \bm{v}_{s}^{\epsilon}|^{2} \leq C_{T}\epsilon.
\end{eqnarray}
Therefore,
\begin{eqnarray}\label{itocon1}
    \mathbb {E}\sup_{0\leq t\leq T}\Big|\int_{0}^{t}(\bm{\sigma}(\bm{x}_{s}^{\epsilon}) {d}\bm{W}_{s}(\epsilon \bm{v}_{s}^{\epsilon})^{*})_{k_{1}k_{2}}\Big|^{2}&\leq&  C_{T}\epsilon,
\end{eqnarray}
similarly,
\begin{eqnarray}\label{itocon2}
\mathbb {E}\sup_{0\leq t\leq T}\Big|\int_{0}^{t}(\epsilon \bm{v}_{s}^{\epsilon}(\bm{\sigma}(\bm{x}_{s}^{\epsilon}) {d}\bm{W}_{s})^{*})_{k_{1}k_{2}}\Big|^{2}&\leq&  C_{T}\epsilon.
\end{eqnarray}
Since $\mid\partial_{x_{l}}\bm{\gamma}_{ij}^{-1}( \bm{x}_{s}^{\epsilon},\mathscr{L}_{\bm{x}_{s}^{\epsilon}})\mid\leq C$ and by Lemma 4.4.2 in \cite{HP:ordieq}
\begin{eqnarray}\label{expbound}
    \mid\int_{0}^{\infty}(e^{-\bm{\gamma}( \bm{x}_{s}^{\epsilon},\mathscr{L}_{\bm{x}_{s}^{\epsilon}})y})_{jk_{1}}(e^{-\bm{\gamma}^{*}( \bm{x}_{s}^{\epsilon},\mathscr{L}_{\bm{x}_{s}^{\epsilon}})y})_{k_{2}l} {d}y\mid
    &\leq&\int_{0}^{\infty}\parallel e^{-\bm{\gamma}( \bm{x}_{s}^{\epsilon},\mathscr{L}_{\bm{x}_{s}^{\epsilon}})y}\parallel\parallel e^{-\bm{\gamma}^{*}( \bm{x}_{s}^{\epsilon},\mathscr{L}_{\bm{x}_{s}^{\epsilon}})y}\parallel {d}y
    \nonumber\\
    &\leq& 2\int_{0}^{\infty} e^{-C_{\lambda_{\bm{\gamma}}}y} {d}y
    \nonumber\\
    &\leq& \frac{2}{C_{\lambda_{\bm{\gamma}}}}  ,
\end{eqnarray}
by Cauchy-schwarz's inequlity, H\"{o}lder inequality and (\ref{itocon1})-(\ref{itocon2}) we have
\begin{eqnarray}\label{i61bound}
   I_{61}  &=&  \mathbb {E}\Big(\sup_{0\leq t\leq T}\Big|(U_{t}^{\epsilon})_{i}\Big|\Big)^{2}
   \nonumber\\
   &=& \mathbb {E}\Big(\sup_{0\leq t\leq T}\Big|
    \int_{0}^{t} [\partial_{x_{l}}\bm{\gamma}_{ij}^{-1}( \bm{x}_{s}^{\epsilon},\mathscr{L}_{\bm{x}_{s}^{\epsilon}})]\times \Big[\int_{0}^{\infty}(e^{-\bm{\gamma}( \bm{x}_{s}^{\epsilon},\mathscr{L}_{\bm{x}_{s}^{\epsilon}})y})_{jk_{1}}(e^{-\bm{\gamma}^{*}( \bm{x}_{s}^{\epsilon},\mathscr{L}_{\bm{x}_{s}^{\epsilon}})y})_{k_{2}l} {d}y \times
    \nonumber\\
    &&\Big((\epsilon \bm{v}_{s}^{\epsilon}F^{*}(\bm{x}_{s}^{\epsilon}))_{k_{1}k_{2}} {d}s+(\epsilon \bm{v}_{s}^{\epsilon}(\bm{\sigma}(\bm{x}_{s}^{\epsilon}) {d}\bm{W}_{s})^{*})_{k_{1}k_{2}}+
    \nonumber\\
    &&(F(\bm{x}_{s}^{\epsilon})(\epsilon \bm{v}_{s}^{\epsilon})^{*})_{k_{1}k_{2}} {d}s+(\bm{\sigma}(\bm{x}_{s}^{\epsilon}) {d}\bm{W}_{s}(\epsilon \bm{v}_{s}^{\epsilon})^{*})_{k_{1}k_{2}}\Big)\Big]\Big|\Big)^{2}
    \nonumber\\
    &\leq& \frac{C}{C_{\lambda_{\bm{\gamma}}}}\Big(\mathbb {E}\sup_{0\leq t\leq T}\Big|
    \int_{0}^{t} (\epsilon \bm{v}_{s}^{\epsilon}F^{*}(\bm{x}_{s}^{\epsilon}))_{k_{1}k_{2}} {d}s\Big|^{2}+\mathbb {E}\sup_{0\leq t\leq T}\Big|\int_{0}^{t}(\epsilon \bm{v}_{s}^{\epsilon}(\bm{\sigma}(\bm{x}_{s}^{\epsilon}) {d}\bm{W}_{s})^{*})_{k_{1}k_{2}}\Big|^{2}+
    \nonumber\\
    &&~~~~~~~~~~~\mathbb {E}\sup_{0\leq t\leq T}\Big|\int_{0}^{t}(F(\bm{x}_{s}^{\epsilon})(\epsilon \bm{v}_{s}^{\epsilon})^{*})_{k_{1}k_{2}}\Big|^{2} {d}s+ \mathbb {E}\sup_{0\leq t\leq T}\Big|\int_{0}^{t}(\bm{\sigma}(\bm{x}_{s}^{\epsilon}) {d}\bm{W}_{s}(\epsilon \bm{v}_{s}^{\epsilon})^{*})_{k_{1}k_{2}}\Big|^{2}\Big)
    \nonumber\\
    &\leq& \frac{C_{T}}{C_{\lambda_{\bm{\gamma}}}}\mathbb {E}
    \int_{0}^{T} \Big|(\epsilon \bm{v}_{s}^{\epsilon}F^{*}(\bm{x}_{s}^{\epsilon}))_{k_{1}k_{2}}\Big|^{2} {d}s+  \mathbb {E}\int_{0}^{T}\Big|(F(\bm{x}_{s}^{\epsilon})(\epsilon \bm{v}_{s}^{\epsilon})^{*})_{k_{1}k_{2}}\Big|^{2} {d}s+C_{T}\epsilon
    \nonumber\\
    &\leq&C_{T}\epsilon.
\end{eqnarray}
By H\"{o}lder inequality, Cauchy-schwarz's inequlity, (\ref{expbound}) and $\mid \partial_{\mu}\bm{\gamma}_{ij}^{-1}( \bm{x}_{s}^{\epsilon},\mathscr{L}_{\bm{x}_{s}^{\epsilon}})(\widetilde{\bm{x}_{s}^{\epsilon}})\mid\leq C$,
\begin{eqnarray}
   I_{62} &=& \mathbb {E}\Big(\sup_{0\leq t\leq T}\Big|\widetilde{\mathbb {E}}\Big[\int_{0}^{t}(\partial_{\mu}\bm{\gamma}_{ij}^{-1}( \bm{x}_{s}^{\epsilon},\mathscr{L}_{\bm{x}_{s}^{\epsilon}})(\widetilde{\bm{x}_{s}^{\epsilon}}))_{l}\times
    \nonumber\\
    &&~~~~~~~~\Big(-\int_{0}^{\infty} (e^{-\bm{\gamma}( \bm{x}_t^{\epsilon},\mathscr{L}_{\bm{x}_t^{\epsilon}})y})_{jk_{1}} (e^{-\bm{\gamma}^{*}( \widetilde{\bm{x}_t^{\epsilon}},\mathscr{L}_{\bm{x}_t^{\epsilon}})y})_{k_{2}l} {d}y
    \nonumber\\
    &&~~~~~~~~~~~~~~~~\times[(\epsilon \bm{v}_{t}^{\epsilon}F^{*}(\widetilde{\bm{x}_t^{\epsilon}}))_{k_{1}k_{2}} {d}t+(\epsilon \bm{v}_{t}^{\epsilon}(\bm{\sigma}(\widetilde{\bm{x}_t^{\epsilon}}) {d}\widetilde{\bm{W}_{t}})^{*})_{k_{1}k_{2}}
    \nonumber\\
    &&~~~~~~~~~~~~~~~~~~~~~+ (F(\bm{x}_t^{\epsilon})(\epsilon \widetilde{\bm{v}_{t}^{\epsilon}})^{*})_{k_{1}k_{2}} {d}t+(\bm{\sigma}(\bm{x}_t^{\epsilon}) {d}\bm{W}_{t}(\epsilon \widetilde{\bm{v}_{t}^{\epsilon}})^{*})_{k_{1}k_{2}}]\Big)\Big]\Big|\Big)^{2}
    \nonumber\\
    &\leq& \frac{C}{C_{\lambda_{\bm{\gamma}}}}\mathbb {E}\Big(\sup_{0\leq t\leq T}\widetilde{\mathbb {E}}\Big|\int_{0}^{t}(\epsilon \bm{v}_{t}^{\epsilon}F^{*}(\widetilde{\bm{x}_t^{\epsilon}}))_{k_{1}k_{2}} {d}t\Big|^{2}+\sup_{0\leq t\leq T}\widetilde{E}\Big|\int_{0}^{t}(\epsilon \bm{v}_{t}^{\epsilon}(\bm{\sigma}(\widetilde{\bm{x}_t^{\epsilon}}) {d}\widetilde{\bm{W}_{t}})^{*})_{k_{1}k_{2}}\Big|^{2}
    \nonumber\\
    &&~~~~~~~~+ \sup_{0\leq t\leq T}\widetilde{\mathbb {E}}\Big|\int_{0}^{t}(F(\bm{x}_t^{\epsilon})(\epsilon \widetilde{\bm{v}_{t}^{\epsilon}})^{*})_{k_{1}k_{2}} {d}t\Big|^{2}+\sup_{0\leq t\leq T}\widetilde{E}\Big|\int_{0}^{t}(\bm{\sigma}(\bm{x}_t^{\epsilon}) {d}\bm{W}_{t}(\epsilon \widetilde{\bm{v}_{t}^{\epsilon}})^{*})_{k_{1}k_{2}}\Big|^{2}\Big)
    \nonumber\\
    &\leq& \frac{C}{C_{\lambda_{\bm{\gamma}}}}\mathbb {E}\Big(\widetilde{\mathbb {E}}\Big|\int_{0}^{T}(\epsilon \bm{v}_{t}^{\epsilon}F^{*}(\widetilde{\bm{x}_t^{\epsilon}}))_{k_{1}k_{2}} {d}t\Big|^{2}+\widetilde{\mathbb {E}}\sup_{0\leq t\leq T}\Big|\int_{0}^{t}(\epsilon \bm{v}_{t}^{\epsilon}(\bm{\sigma}(\widetilde{\bm{x}_t^{\epsilon}}) {d}\widetilde{\bm{W}_{t}})^{*})_{k_{1}k_{2}}\Big|^{2}
    \nonumber\\
    &&~~~~~~~~+ \widetilde{\mathbb {E}}\Big|\int_{0}^{T}(F(\bm{x}_t^{\epsilon})(\epsilon \widetilde{\bm{v}_{t}^{\epsilon}})^{*})_{k_{1}k_{2}} {d}t\Big|^{2}+\widetilde{\mathbb {E}}\sup_{0\leq t\leq T}\Big|\int_{0}^{t}(\bm{\sigma}(\bm{x}_t^{\epsilon}) {d}\bm{W}_{t}(\epsilon \widetilde{\bm{v}_{t}^{\epsilon}})^{*})_{k_{1}k_{2}}\Big|^{2}\Big),
    \nonumber
\end{eqnarray}
hence, by H\"{o}lder inequality, Lemma \ref{lem1}, assumption \ref{a2}, Doob's maximal inequality, Burkh\"{o}lder inequality, we have
\begin{eqnarray}\label{i62bound}
    I_{62} &\leq& \frac{C}{C_{\lambda_{\bm{\gamma}}}}\mathbb {E}\Big(\widetilde{\mathbb {E}}\int_{0}^{T}\Big|(\epsilon \bm{v}_{t}^{\epsilon}F^{*}(\widetilde{\bm{x}_t^{\epsilon}}))_{k_{1}k_{2}}\Big|^{2} {d}t+\widetilde{\mathbb {E}}\sup_{0\leq t\leq T}\Big|\int_{0}^{t}(\epsilon \bm{v}_{t}^{\epsilon}(\bm{\sigma}(\widetilde{\bm{x}_t^{\epsilon}}) {d}\widetilde{\bm{W}_{t}})^{*})_{k_{1}k_{2}}\Big|^{2}
    \nonumber\\
    &&+ \widetilde{\mathbb {E}}\int_{0}^{T}\Big|(F(\bm{x}_t^{\epsilon})(\epsilon \widetilde{\bm{v}_{t}^{\epsilon}})^{*})_{k_{1}k_{2}}\Big|^{2} {d}t+\widetilde{\mathbb {E}}\sup_{0\leq t\leq T}\Big|\int_{0}^{t}(\bm{\sigma}(\bm{x}_t^{\epsilon}) {d}\bm{W}_{t}(\epsilon \widetilde{\bm{v}_{t}^{\epsilon}})^{*})_{k_{1}k_{2}}\Big|^{2}\Big)
    \nonumber\\
    &\leq& \frac{C}{C_{\lambda_{\bm{\gamma}}}}\Big(\int_{0}^{T}\mathbb {E}|\epsilon \bm{v}_{t}^{\epsilon}|^{2}\widetilde{\mathbb {E}}|F^{*}(\widetilde{\bm{x}_t^{\epsilon}})|^{2} {d}t+\int_{0}^{T}E|F(\bm{x}_t^{\epsilon})|^{2}\widetilde{\mathbb {E}}\mid(\epsilon \widetilde{\bm{v}_{t}^{\epsilon}})^{*}\mid^{2} {d}t
    \nonumber\\
    &&+ 4E(\widetilde{\mathbb {E}}\Big|\int_{0}^{T}(\epsilon \bm{v}_{t}^{\epsilon}(\bm{\sigma}(\widetilde{\bm{x}_t^{\epsilon}}) {d}\widetilde{\bm{W}_{t}})^{*})_{k_{1}k_{2}}\Big|^{2})+\widetilde{\mathbb {E}}(\mathbb {E}\Big|\int_{0}^{T}(\bm{\sigma}(\bm{x}_t^{\epsilon}) {d}\bm{W}_{t}(\epsilon \widetilde{\bm{v}_{t}^{\epsilon}})^{*})_{k_{1}k_{2}}\Big|^{2})
    \nonumber\\
     &\leq& \frac{C}{C_{\lambda_{\bm{\gamma}}}}\Big(C_{T}\epsilon
    + 4\mathbb {E}(\widetilde{\mathbb {E}}\int_{0}^{T}|\epsilon \bm{v}_{t}^{\epsilon}|^{2}|(\bm{\sigma}(\widetilde{\bm{x}_t^{\epsilon}}))^{*}|^{2} {d}t)+\widetilde{\mathbb {E}}(\mathbb {E}\int_{0}^{T}|(\bm{\sigma}(\bm{x}_t^{\epsilon})|^{2}|(\epsilon \widetilde{\bm{v}_{t}^{\epsilon}})^{*}|^{2} {d}t)
    \nonumber\\
    &\leq& \frac{C_{T}}{C_{\lambda_{\bm{\gamma}}}}\epsilon.
\end{eqnarray}
Therefore, combining (\ref{i61bound}) and (\ref{i62bound}) we have
\begin{eqnarray}
    I_{6} &:=& I_{61}+I_{62}\leq  C_{T}\epsilon.
\end{eqnarray}
 {For $I_{7}$},
\begin{eqnarray}
   {I_{7}} &\leq&  2 \mathbb {E}\sup_{0\leq t\leq T} \Big|\int_{0}^{t} [\partial_{x_{l}}\bm{\gamma}_{ij}^{-1}( \bm{x}_{s}^{\epsilon},\mathscr{L}_{\bm{x}_{s}^{\epsilon}})]\times
    \nonumber\\
    &&~~~~~~~~\Big[-\int_{0}^{\infty}(e^{-\bm{\gamma}( \bm{x}_{s}^{\epsilon},\mathscr{L}_{\bm{x}_{s}^{\epsilon}})y})_{jk_{1}}(e^{-\bm{\gamma}^{*}( \bm{x}_{s}^{\epsilon},\mathscr{L}_{\bm{x}_{s}^{\epsilon}})y})_{k_{2}l} {d}y\Big] {d}[(\epsilon \bm{v}_{s}^{\epsilon})_{k_{1}}(\epsilon \bm{v}_{s}^{\epsilon})_{k_{2}}]\Big|^2
    \nonumber\\
    &&~~~~+ 2 \mathbb {E}\sup_{0\leq t\leq T}  \Big[ \widetilde{\mathbb {E}}\int_{0}^{t}(\partial_{\mu}\bm{\gamma}_{ij}^{-1}( \bm{x}_{s}^{\epsilon},\mathscr{L}_{\bm{x}_{s}^{\epsilon}})(\widetilde{\bm{x}_{s}^{\epsilon}}))_{l}\times
    \nonumber\\
    &&~~~~~~~~\Big(-\int_{0}^{\infty} (e^{-\bm{\gamma}( \bm{x}_t^{\epsilon},\mathscr{L}_{\bm{x}_t^{\epsilon}})y})_{jk_{1}}(e^{-\bm{\gamma}^{*}( \widetilde{\bm{x}_t^{\epsilon}},\mathscr{L}_{\bm{x}_t^{\epsilon}})y})_{k_{2}l} {d}y {d}[(\epsilon \bm{v}_{t}^{\epsilon})_{k_{1}}(\epsilon \widetilde{\bm{v}_{t}^{\epsilon}})^{*}_{k_{2}}]\Big)\Big]^{2}\nonumber\\
     &:=& I_{71}+I_{72}.
\end{eqnarray}
For $I_{71}$, using (\ref{expbound}), assumption \ref{a3} and Lemma \ref{pconver},
$$I_{71} \leq \frac{8C^2}{C^2_{\lambda_{\bm{\gamma}}}} \mathbb {E}\sup_{0\leq t\leq T}  \Big|\int_{0}^{t} {d}[(\epsilon \bm{v}_{s}^{\epsilon}) (\epsilon \bm{v}_{s}^{\epsilon}) ]\Big|^2
\leq \frac{8C^2}{C^2_{\lambda_{\bm{\gamma}}}} \mathbb {E}\sup_{0\leq t\leq T} [\epsilon^2 |\bm{v}_{t}^{\epsilon}|^2- \epsilon^2|v_0|^2] \leq C_T \sqrt{\epsilon }.$$
For $I_{72}$, defining
\begin{eqnarray}
     &&H^\epsilon_t := \widetilde{\mathbb {E}}\Big[\int_{0}^{t}(\partial_{\mu}\bm{\gamma}_{ij}^{-1}( \bm{x}_{s}^{\epsilon},\mathscr{L}_{\bm{x}_{s}^{\epsilon}})(\widetilde{\bm{x}_{s}^{\epsilon}}))_{l}\times
    \nonumber\\
    \Big(&-&\int_{0}^{\infty} (e^{-\bm{\gamma}( \bm{x}_{s}^{\epsilon},\mathscr{L}_{\bm{x}_{s}^{\epsilon}})y})_{jk_{1}}(e^{-\bm{\gamma}^{*}( \widetilde{\bm{x}_{s}^{\epsilon}},\mathscr{L}_{\bm{x}_{s}^{\epsilon}})y})_{k_{2}l} {d}y {d}[(\epsilon \bm{v}_{s}^{\epsilon})_{k_{1}}(\epsilon \widetilde{\bm{v}_{s}^{\epsilon}})^{*}_{k_{2}}]\Big)\Big].
    \nonumber
\end{eqnarray}
    \nonumber
Considering the product space $(\Omega\otimes\Omega,\mathscr{F}\otimes\mathscr{F},P\otimes P)$ and
defining
$$\widetilde{\bm{x}_{s}^{\epsilon}}(\omega,\widetilde{\omega}):=\bm{x}_{s}^{\epsilon}(\omega),~~\widetilde{\widetilde{\bm{x}_{s}^{\epsilon}}}(\omega,\widetilde{\omega}):=\widetilde{\bm{x}_{s}^{\epsilon}}(\widetilde{\omega}),$$
we still use $\bm{x}_{s}^{\epsilon}$ and $\widetilde{\bm{x}_{s}^{\epsilon}}$ to denote $\widetilde{\bm{x}_{s}^{\epsilon}}(\omega,\widetilde{\omega})$ and $\widetilde{\widetilde{\bm{x}_{s}^{\epsilon}}}(\omega,\widetilde{\omega})$ for convenience. Furthermore, defining
\begin{eqnarray}
    f(\bm{x}_{s}^{\epsilon},\widetilde{\bm{x}_{s}^{\epsilon}},\mathscr{L}_{\bm{x}_{s}^{\epsilon}}) &=& (\partial_{\mu}\bm{\gamma}_{ij}^{-1}( \bm{x}_{s}^{\epsilon},\mathscr{L}_{\bm{x}_{s}^{\epsilon}})(\widetilde{\bm{x}_{s}^{\epsilon}}))_{l}\times\Big(-\int_{0}^{\infty} (e^{-\bm{\gamma}( \bm{x}_{s}^{\epsilon},\mathscr{L}_{\bm{x}_{s}^{\epsilon}})y})_{jk_{1}}(e^{-\bm{\gamma}^{*}( \widetilde{\bm{x}_{s}^{\epsilon}},\mathscr{L}_{\bm{x}_{s}^{\epsilon}})y})_{k_{2}l} {d}y,
    \nonumber
\end{eqnarray}
using assumption \ref{a3} and (\ref{expbound}), we get
$$    |f(\bm{x}_{s}^{\epsilon},\widetilde{\bm{x}_{s}^{\epsilon}},\mathscr{L}_{\bm{x}_{s}^{\epsilon}})|^2 \leq C  .$$
Then
\begin{eqnarray}
 \mathbb {E}\sup_{0\leq t\leq T} H^\epsilon_t \leq  C\mathbb {E}\sup_{0\leq t\leq T}  \widetilde{\mathbb {E}}  \int_{0}^{t} {d}[(\epsilon \bm{v}_{s}^{\epsilon})_{k_{1}}(\epsilon \widetilde{\bm{v}_{s}^{\epsilon}})^{*}_{k_{2}}].
\end{eqnarray}
Note that
\begin{eqnarray}
 (\mathbb {E}\times \widetilde{\mathbb {E}})\Big(\sup_{0\leq t\leq T} \mid \epsilon \bm{v}_{t}^{\epsilon}\epsilon\widetilde{\bm{v}_{t}^{\epsilon}} \mid\Big)^{2}&\leq&\frac{1}{2}\Big((\mathbb {E}\times \widetilde{\mathbb {E}})\Big(\sup_{0\leq t\leq T} \mid \epsilon \bm{v}_{t}^{\epsilon} \mid\Big)^{4}+(\mathbb {E}\times \widetilde{\mathbb {E}})\Big(\sup_{0\leq t\leq T} \mid \epsilon\widetilde{\bm{v}_{t}^{\epsilon}} \mid\Big)^{4}\Big)
    \nonumber\\
    &\leq&\frac{1}{2}\Big(\mathbb {E}\Big(\sup_{0\leq t\leq T} \mid \epsilon \bm{v}_{t}^{\epsilon} \mid\Big)^{4}+\widetilde{\mathbb {E}}\Big(\sup_{0\leq t\leq T} \mid \epsilon\widetilde{\bm{v}_{t}^{\epsilon}} \mid\Big)^{4}\Big)
    \nonumber\\
    &\leq& C\epsilon,
    \nonumber
\end{eqnarray}
where $(\mathbb {E}\times \widetilde{\mathbb {E}})[\cdot] = \int_{\Omega\times\widetilde{\Omega}}(\cdot) {d}(P\times\widetilde{P})$.
Then
$$I_{72}  \leq \epsilon C_T, $$
and
$$I_{7}  \leq \sqrt{\epsilon} C_T. $$
Therefore,
\begin{eqnarray}
   \mathbb {E}\sup_{0\leq t\leq T}\mid\bm{x}_t^{\epsilon} -\bm{x}_{t} \mid^{2} \leq   C \sqrt{\epsilon} + C_T \int_{0}^{T}\mathbb {E}\sup_{0\leq r\leq t}|\bm{x}_{r}^{\epsilon}-\bm{x}_{r}^{\epsilon}|^{2} {d}t,
\end{eqnarray}
using Gronwall inequality
$$    \mathbb {E}\sup_{0\leq t\leq T}\mid\bm{x}_t^{\epsilon} -\bm{x}_{t} \mid^{2} \leq C\sqrt{\epsilon} ,$$
that is
$$  \lim_{\epsilon \rightarrow 0 }  \mathbb {E}\sup_{0\leq t\leq T}\mid\bm{x}_t^{\epsilon} -\bm{x}_{t} \mid^{2}  = 0 .$$

\end{proof}

\section{Extension}
  \setcounter{equation}{0}
  \renewcommand{\theequation}
{5.\arabic{equation}}
We first remark a extension of the results in this paper by setting general nonlinear reaction and noise  by $ \bm{F}( \bm{x}_{s}^{\epsilon},\mathscr{L}_{\bm{x}_{s}^{\epsilon}}) $  and  $\bm{\sigma}(\bm{x}_t^{\epsilon}, \mathscr{L}_{\bm{x}_t^{\epsilon}}) $, in which $\bm{F}$ and $\bm{\sigma}$ are state-dependent and distribution-dependent.
\begin{eqnarray}\label{mainextension}
\left\{\begin{array}{ll}
  {d} \bm{x}_t^{\epsilon} = \bm{v}_{t}^{\epsilon} {d}t,\\
  {d}\bm{v}_{t}^{\epsilon} = \frac{1}{\epsilon}F(\bm{x}_t^{\epsilon}, \mathscr{L}_{\bm{x}_t^{\epsilon}}) {d}t-\frac{1}{\epsilon}\bm{\gamma}( \bm{x}_t^{\epsilon},\mathscr{L}_{\bm{x}_t^{\epsilon}}) \bm{v}_{t}^{\epsilon} {d}t+\frac{1}{\epsilon}\bm{\sigma}(\bm{x}_t^{\epsilon}, \mathscr{L}_{\bm{x}_t^{\epsilon}}) {d}\bm{W}_{t},
\end{array}
\right.
\end{eqnarray}
Suppose that assumptions \ref{a1}-\ref{a4} still hold,  and we rewrite assumptions \ref{a1}-\ref{a2} as follows,

\begin{assumption}\label{a1-1}
\begin{eqnarray}\label{e1-2}
\mid \bm{F}(\bm{x}_{1}, \mu_1)-\bm{F}(\bm{x}_{2}, \mu_2)\mid+\parallel \bm{\sigma}(\bm{x}_{1}, \mu_1)-\bm{\sigma}(\bm{x}_{2}, \mu_2)\parallel
\leq C_{T}\Big[\mid \bm{x}_{1}-\bm{x}_{2}\mid +\mathbb{W}(\mu_{1},\mu_{2})\Big] ;
\end{eqnarray}
\end{assumption}

\begin{assumption}\label{a2-2}
\begin{eqnarray}\label{e2-2}
\mid \bm{F}(\bm{x}, \mu)\mid^{2}+\parallel \bm{\sigma}(\bm{x}, \mu)\parallel^{2}
\leq C_{T}\Big[1+\mid \bm{x}\mid^{2} + \mu(\cdot)\Big];
\end{eqnarray}
\end{assumption}
Now, for the system~(\ref{mainextension}), we have the following result.

\begin{theorem}\label{thm-extend}
    Assume that assumptions \ref{a3}-\ref{a4}  and  assumptions \ref{a1-1}-\ref{a2-2}   hold. $\bm{x}_t^{\epsilon}$ is the solution of the equation~(\ref{mainextension}),  for any $T>0$, and $\bm{x}_0,\bm{v}_0\in \mathbb{R}^{d}$,
    \begin{eqnarray}
        \mathbb{E}(\sup_{0\leq t\leq T}\mid \bm{x}_t^{\epsilon}-\bm{x}_{t}\mid^2)\leq C\sqrt{\epsilon},
    \end{eqnarray}
    where $C$ is a constant depending on $T$ , $\bm{x}_0$ and $\bm{v}_0$.
     Here, $\bm{x}_{t}$ is the solution of the following limiting equation
$$
  {d}\bm{x}_{t} = [\bm{\gamma}^{-1}(\bm{x}_{t},\mathscr{L}_{\bm{x}_{t}})F(\bm{x}_{t},\mathscr{L}_{\bm{x}_{t}})+S( \bm{x}_{t},\mathscr{L}_{\bm{x}_{t}}) +\widetilde{S}(\bm{x}_{t},\mathscr{L}_{\bm{x}_{t}})] {d}t+\bm{\gamma}^{-1}( \bm{x}_{t},\mathscr{L}_{\bm{x}_{t}}) \bm{\sigma}(\bm{x}_{t}, \mathscr{L}_{\bm{x}_{t}}) {d}\bm{W}_{t}
$$
where $S_{i}(\bm{x},\mu) = \frac{\partial}{\partial x_{l}}[\bm{\gamma}^{-1}_{ij}(\bm{x},\mu)]J_{jl}(\bm{x},\mu)$ and $J$ solves the Lyapunov equation
$$
    J(\bm{x},\mu)\bm{\gamma}^{*}(\bm{x},\mu)+\bm{\gamma}(\bm{x},\mu)J(\bm{x},\mu) = \bm{\sigma}(\bm{x}, \mu)\bm{\sigma}^{*}(\bm{x}, \mu),
$$
and $\widetilde{S}_{i}(\bm{x},\mu) = \widetilde{E}[(\partial_{\mu}\bm{\gamma}^{-1}_{ij}(\bm{x},\mu)(\widetilde{\bm{x}}))_{l}\widetilde{J}_{jl}(\bm{x},\widetilde{\bm{x}},\mu)]$ and $\widetilde{J}$ solves the Sylevster equation
$$
    \widetilde{J}(\bm{x},\widetilde{\bm{x}},\mu)\bm{\gamma}^{*}(\widetilde{\bm{x}},\mu)+\bm{\gamma}(\bm{x},\mu)\widetilde{J}(\bm{x},\widetilde{\bm{x}},\mu) = \bm{\sigma}(\bm{x}, \mu)\bm{\sigma}^{*}(\widetilde{\bm{x}}, \mu).
$$
\end{theorem}
\begin{proof}
Using the same method as Theorem \ref{mainthm}, it is easily to obtain this result.

\end{proof}


\begin{thebibliography}{00}\addtolength{\itemsep}{-1.5ex}


\bibitem{BLPR:mestdi}
Buckdahn, R.,  Li, J.,  Peng, S., and  Rainer, C., Mean-field stochastic differential equations and associated pdes. {\em Stochastic Analysis and Applications}, 35(3), 542-568, 2014.

\bibitem{carp:nomefiga}
Cardaliaguet, P., {\em Notes on Mean Field Games}. 2012.

\bibitem{Carrillo} Carrillo,J. A., Choi, Y.-P.,  Mean--field limits: From particle descriptions to macroscopic equations,{\em Arch. Ration. Mech. Anal.} {\bf241}(3), 1529--1573, 2021.

\bibitem{Cerrai}
Cerrai, S., Freidlin, M., Smoluchowski-Kramers approximation for a general class of SPDE's.  {\em Journal of Evolution Equations}, 6,  657-689, 2006.






\bibitem{MW}  Freidlin, M., Hu, W., Smoluchowski--Kramers approximation in the case of variable friction, {\em Journal of Mathematical Sciences} {\bf179}, 184--207, 2011.

\bibitem{Freidlin}
Freidlin, M., Some remarks on the Smoluchowski-Kramers approximation. {\em Journal of Statistical Physics}, 117,  617-634, 2004.


\bibitem{Hanggi}
 H\"{a}nggi, P., Nonlinear fluctuations: the problem of deterministic limit and reconstruction of stochastic dynamics. {\em Physical. Review.}, A {{25}}, 1130-1136, 1982.



\bibitem{HvS:mcsdme}
Hammersley, W.R., \u{S}i\u{S}ka, D., and Szpruch, {\L}. McKean Vlasov SDEs under measure dependent Lyapunov conditions. arXiv: Probability. 2018.

\bibitem{HP:ordieq}
Hartman, P.,  Ordinary Differential Equations. {\em Journal of the American Statistical Association}, 60, 934, 1965.

\bibitem{Herzog}
Herzog, D.,  Hottovy, S., Volpe, G.,  The small-mass limit for Langevin dynamics with unbounded coefficients and positive friction.  {\em Journal of Statistical Physics},  163, 659-673, 2016.

\bibitem{shamgv:smklim}
Hottovy, S., McDaniel, A., Volpe, G. et al.,  The Smoluchowski-Kramers Limit of Stochastic Differential Equations with Arbitrary State-Dependent Friction. {\em Communications in Mathematical  Physics}, 336, 1259-1283 ,2015.

\bibitem{Kramers}
Kramers, H.A., Brownian motion in a field of force and the diffusion model of chemical reactions. {\em Physica}, {{7}}, 284-304, 1940.

\bibitem{Kurtz}
Kurtz, T.G., Protter, P., Weak limit theorems for stochastic integrals and stochastic differential equations. {\em The Annals of Probability}, 19, 1035-1070, 1991.


\bibitem{Xr.L} Liu, X.R., Wang, W., Small mass limit in the propagation of chaos  for interacting particles system with common noise, preprint, 2022.


\bibitem{Nelson}
Nelson, E., {\em Dynamical Theories of Brownian Motion}. Princeton University Press, Princeton, 1967.





\bibitem{Papanicolaou}
Papanicolaou, A.,  Filtering for fast mean-reverting processes.  {\em Asymptotic Analysis}, 70, 155-176, 2010.



\bibitem{Pardoux}
Pardoux, \`{E}., Veretennikov, AY, On Poisson equation and diffusion approximation. {\em  The Annals of Probability}, 31, 3, 1166-1192, 2003.

\bibitem{Sancho}
Sancho, J.M., San Miguel, M., D\"{u}rr, M., Adiabatic elimination for systems of Brownian particles with
nonconstant damping coefficients. {\em Journal of  Statistical Physics}. 28, 291-305, 1982.

\bibitem{Smoluchowski}
Smoluchowski, M., Drei vortrage \"{u}ber diffusion brownsche bewegung and koagulation von kolloidteilchen. {\em Phys.Z.}, {{17}}, 557-585, 1916.


\bibitem{Spiliopoulos}
Spiliopoulos, K., A note on the Smoluchowski-Kramers approximation for the Langevin equation with reflection.  {\em Stochastics and Dynamics}, 7, 141-152, 2007.

\bibitem{W.Wang}
Wang, W., Lv, G.Y., Wei, J.L., Small mass limit in mean field theory for stochastic $N$ particle system, {\em Journal of  Mathematical  Physics}, 63, 083302. 2022.
\bibitem{W.Wang1}
Wang, M.M., Su. D. and Wang, W., Averaging on macroscopic scales with application to Smoluchowski--Kramers approximation. {\em Journal of Statistical Physics}, {\bf191}(22), 1-17,2024.

\bibitem{DDS}
Xu, D.D, Zhang, D.J and Zhao, S.L., The Sylvester equation and integrable equations: I. The Korteweg-de Vries system and sine-Gordon equation,  {\em Journal of Nonlinear Mathematical Physics}, 21:3, 382-406, 2014.






\end{thebibliography}
\end{document}